\documentclass{article}
\usepackage[utf8]{inputenc}

\usepackage{fullpage}
\usepackage{amsfonts}
\usepackage{amsmath}
\usepackage{amssymb}
\usepackage{amsthm}
\usepackage{cite}

\usepackage{graphics}
\usepackage{graphicx}

\usepackage{hyperref}

\usepackage{subcaption}

\usepackage{latexsym,amsfonts,amssymb,amsmath,amsthm}

\usepackage{booktabs} 
\usepackage{graphicx}
\usepackage{pgfplots}
\usepackage[all]{nowidow}
\usepackage{soul}
\usepackage[utf8]{inputenc}
\usepackage{tikz}
\usetikzlibrary{automata}
\usepackage{multicol,comment}
\usepackage{algpseudocode,algorithm,algorithmicx}
\usepackage{cite}
\usepackage{hyperref}
\usepackage[inline]{enumitem} 

\usepackage{ulem}

\definecolor{blue}{HTML}{1F77B4}
\definecolor{orange}{HTML}{FF7F0E}
\definecolor{green}{HTML}{2CA02C}

\pgfplotsset{compat=1.14}

\newtheorem{tm}{Theorem}[section]
\newtheorem{prop}[tm]{Proposition}

\newtheorem{coro}[tm]{Corollary}
\newtheorem{exa}[tm]{Example}
\newtheorem{lem}[tm]{Lemma}

\newtheorem{rk}[tm]{Remark}

\numberwithin{equation}{section}
\numberwithin{tm}{section}

\title{Global dynamics of a  two-stage structured diffusive population model in time-periodic and spatially heterogeneous environments
}
\author{ H. M. Gueguezo\footnote{marcos.gueguezo@imsp-uac.org; Institut de Math\'ematiques et de Sciences Physiques, Universit\'e d'Abomey-Calavi,  Abomey-Calavi, Benin },\and  T. J. Doumat\`e\footnote{jonas.doumate@uac.bj; D\'epartement de Math\'ematiques,
		Universit\'e d'Abomey-Calavi,  Abomey-Calavi, Benin},  
	\and R. B. Salako \footnote{rachidi.salako@unlv.edu; Department of Mathematical Sciences, University of Nevada Las Vegas, NV 89154, USA
	}
}

\date{} 

\begin{document}
	
	\maketitle

	\begin{abstract}
		This work examines the global dynamics of classical solutions of a two-stage (juvenile-adult) reaction-diffusion population model in time-periodic and spatially heterogeneous environments. It is shown that the sign of the principal eigenvalue $\lambda_*$ of the time-periodic linearized system at the trivial solution completely determines the persistence of the species. Moreover, when $\lambda_*>0$, there is at least one time-periodic positive entire solution. A fairly general sufficient condition ensuring the uniqueness and  global stability of the positive time-periodic solution is obtained. In particular, 
		classical solutions eventually stabilize at the unique time-periodic positive solutions if either each subgroup's  intra-stage  growth  and  inter-stage  competition rates are proportional, or  the environment is temporally homogeneous and both subgroups diffuse slowly. In the later scenario, the asymptotic profile of steady states   with respect to small diffusion rates is established.

	\end{abstract}
	
	
	\smallskip
	
	
	\section{Introduction}
	
	The investigation of population dynamics through reaction-diffusion models occur in several evolution systems arising in applied sciences \cite{Dockery1998,DS2022,Hasting1983,He2016-1,He2016-2,HMMV2003,Lou2008,LS2020
	}. These studies have found numerous applications in several disciplines such as  ecology, biology, epidemiology \cite{brauer2012mathematical,CC2003,DNS2023,AM1979,C2014,HMMV2003,JSM1976,LS2022,Ni2011,Salako2023}.  There are several  approaches to modelling internal dispersal.  For example, it is customary to use the Laplace operator to model the dispersal mechanism of species exhibiting local and random movements between adjacent locations. 
	Moreover, in most studies, it is common to assume that all members of the population 
	diffuse uniformly at the same rate.  However,  there may be individual variation in  dispersal  mechanisms and/or diffusion rates within the same population.  Such  variations  may occur at different stages of maturity in the life cycle of the population \cite{BZ2003,AM2000,CCM2020,CCY2008}.  In the current work, we study the dynamics of a population where the adults and juveniles adopt the same dispersal mechanism but differ in their diffusion rates in time-periodic and spatially heterogeneous environments.
	
	\medskip
	
	Consider the system of parabolic equations 
	\begin{equation}\label{Eq1}
		\begin{cases}
			\partial_tu_1=d_1\Delta u_1 +r(t,x)u_2 -s(t,x)u_1-(a(t,x)+b(t,x)u_1+c(t,x)u_2)u_1 & x\in\Omega,\ t>0,\cr 
			\partial_tu_2=d_2\Delta u_2 +s(t,x)u_1 -(e(t,x)+f(t,x)u_2+g(t,x)u_1)u_2 & x\in\Omega,\ t>0,\cr 
			0=\partial_{\vec{n}}u_1=\partial_{\vec{n}}u_2 & x\in\partial\Omega,\ t>0,
		\end{cases}
	\end{equation}
	where $r,s,a,b,c,e,f \text{ and } g$ are nonnegative functions. The function 
	${\bf u}(t,x):=(u_1(t,x),u_2(t,x))$ denotes the density function of a population:  $u_1(t,x)$ and $u_2(t,x)$ are the density functions of the immature juveniles and adults who have attained reproductive maturity, respectively. $\Omega$ is a bounded open domain  in $\mathbb{R}^n$ with a smooth boundary $\partial\Omega$. ${\vec{n}}$ is the outward normal unit vector at $\partial\Omega$. $r(t,x)$ (resp. $s(t,x)$) is the adults' (resp. juvenile's) local and temporal reproduction (resp. maturity) rate. $e(t,x)$ (resp. $a(t,x)$) is the local and temporal death rate of the adults (resp. juveniles). The functions $b(t,x)$, $c(t,x)$, $f(t,x)$ and $g(t,x)$ account for the local and temporal limitations due to  over crowding and inter-specific competitions between the members of the populations. The positive constants $d_1$ and $d_2$ are the diffusion rates of the juveniles and adults, respectively. The homogeneous Neumann boundary condition is imposed in \eqref{Eq1} to indicate that the population lives on an isolated habitat and there is no movement across its boundary. The aim of the current work is to study the dynamics  of ``classical solutions" to \eqref{Eq1}. We say that  a nonnegative function ${\bf u}(t,x)=(u_1(t,x),u_2(t,x))$ defined on $(t_0, t_0+\tilde{T})\times \bar{\Omega}$, for some $\tilde{T}>0$ and $t_0\in\mathbb{R}$, is a classical solution of \eqref{Eq1}, if ${\bf u}\in C^{1,2}((t_0,t_0+\tilde{T})\times \Omega)\cap C^{0,1}((t_0,t_0+\tilde{T})\times\bar{\Omega})$ and satisfies system \eqref{Eq1} in the classical sense. Since $u_1(t,x)$ and $u_2(t,x)$ in \eqref{Eq1} model the density functions of some population, then we shall only be concerned with nonnegative classical solutions of \eqref{Eq1}.
	
	\medskip
	
	System \eqref{Eq1} has been recently studied by several authors when the environment is temporally constant (see \cite{AM2000,bouguima2012asymptotic, BZ2003,CCM2020,CCS2023,CMM2001,Henaoui2012,OSUU2023,OSUU2024,xu2023age} and the references cited therein). Assuming that the coefficients are time-independent:  the works \cite{AM2000, Henaoui2012} examined the dynamics of solutions of system \eqref{Eq1} with the homogeneous Dirichlet boundary conditions, and established some  results on the existence and uniqueness of a positive  steady-state solution; the work \cite{bouguima2012asymptotic} examined the uniqueness and global stability of the positive constant equilibrium solution whenever it exists under the additional assumptions $b=c$ and $f=g$ when all the coefficients are both temporally and spatially homogeneous; the works \cite{BZ2003,CCM2020} obtained necessary and sufficient conditions  for the persistence of the population, and the existence of positive steady-state solutions of system \eqref{Eq1}; the effects of dispersal rates on the persistence and spatial distributions of the steady-state solutions of \eqref{Eq1} are studied in \cite{CCM2020,CCS2023}.  The authors of \cite{OSUU2023,OSUU2024} studied system \eqref{Eq1} with nonlocal dispersal operators in temporally constant environments and showed that the population eventually goes extinct if and only if the principal spectrum point of its linearization  at the trivial solution is less or equal to zero. 
	
	\medskip 
	
	Thanks to the works cited above, it is known that the sign of the principal eigenvalue of the linearized system at the trivial solution plays an important role on the dynamics of solutions to \eqref{Eq1} in temporally homogeneous environments. An important  question partially answered by previous studies is concerned with the uniqueness and stability of steady-state solution of \eqref{Eq1}. Indeed, it follows from the reaction term in \eqref{Eq1} (mainly the terms $(r-cu_1)u_2$ in the first equation, and  $(s-gu_2)u_1$ in the second equation) that the solution operator generated by nonnegative solutions of \eqref{Eq1} is cooperative when population density is small, and competitive when population density is large. This feature makes the study of the uniqueness and global stability of positive entire solutions of \eqref{Eq1} more delicate as most standard arguments from the existing literature do not apply. Hence, previous attempts have been concerned with identifying sufficient conditions on the parameters of the model \eqref{Eq1} which guarantee the existence, uniqueness and global stability of the positive steady-state solutions.
	
	\medskip
	
	In the current work, the question of uniqueness and globally stability of the positive steady state solution of \eqref{Eq1} when the environment is temporally homogeneous is completely solved when the population diffuses slowly (see Theorem \ref{TH5}).  Moreover, the spatial profiles of the steady state solutions with respect to small population diffusion rates is obtained (see Theorem \ref{TH6}). We note that some partial results are  known on the spatial distributions of the positive steady with respect to small diffusion rates of the populations when the environments is temporally constants (see \cite[Theorem 1]{CCM2020} and \cite[Theorem 2.4]{CCS2023}).  
	
	\medskip
	
	\vspace{0.1 in}
	Supposing that the environment is  spatially heterogeneous and depends periodically in time, Theorems \ref{TH1} and \ref{TH2} identify  the necessary and sufficient conditions  for the extinction and persistence of the species, respectively.  Moreover, when the species persists, the existence of a strictly positive entire solution is established in Theorem \ref{TH3}. Sufficient conditions for the uniqueness and global stability of the positive entire solution is considered in  Theorem \ref{TH4}. In particular, when the ratios $c/r$ and $g/s$ are constant,  positive entire solution of \eqref{Eq1}, whenever  exists, is unique and globally stable (see Corollary \ref{Cor1}). 
	
	\medskip
	
	The rest of our work is organized as follows. Section \ref{sec2} introduces some notations and definitions necessary for the clarity of our exposition. This section is concluded with the statement of our main results. We present  the proofs of our main results 
	in Section \ref{sec3}.
	
	\section{Notations, Definitions and Main Results}\label{sec2}
	
	\subsection{Notations and Definitions}
	
	Let $C(\overline{\Omega})$ denote the Banach space of uniformly continuous functions on $\overline{\Omega}$ endowed with the standard uniform-topology norm,
	$$
	\|w\|_{\infty}:=\max_{x\in\bar{\Omega}}|w(x)|\quad w\in C(\bar{\Omega}).
	$$ 
	Define the sets 
	$$
	C^+(\overline{\Omega}):=\{w\in C(\overline{\Omega}):\ w(x)\ge 0 \ \text{for all}\ x\in\overline{\Omega}\}
	$$ 
	and 
	$$ 
	C^{++}(\overline{\Omega})=\{w\in C^{+}(\overline{\Omega}):\ w(x)>0\ \text{for all}\ x\in\overline{\Omega}\}.
	$$
	Clearly, $C^{+}(\overline{\Omega})$ (resp. $C^{++}(\overline{\Omega})$) is a closed (resp. open) subset of $C(\overline{\Omega})$. Next, fix $T>0$ and let $\mathcal{X}_T$ denote the Banach space
	$$ 
	\mathcal{X}_T:=\{w : \mathbb{R}\to C(\overline{\Omega}) \ \text{such that }\ w\ \text{is continuous and }\ T\ \text{periodic} \}
	$$
	endowed with the sup-norm 
	$$ 
	\|w\|_{\mathcal{X}_T}=\max_{t\in[0,T]}\|w(t,\cdot)\|_{\infty}\quad\ w\in\mathcal{X}_T.
	$$
	Similarly we introduce the sets 
	$$ 
	\mathcal{X}_T^+:=\{w\in \mathcal{X}_T:\ w(t,\cdot)\in C^+(\overline{\Omega})\ \text{for all}\ t\in[0,T]\}$$ 
	and
	$$ \mathcal{X}_T^{++}:=\{w\in\mathcal{X}_T^{+}:\ w(t,\cdot)\in C^{++}(\overline{\Omega})\ \text{for all}\ t\in[0,T]\}.
	$$
	Hence the sets $\mathcal{X}_T^{+}$ and $\mathcal{X}_T^{++}$ are closed and open subsets of  $\mathcal{X}_T$, respectively. Moreover, $\mathcal{X}_T^{+}$ induces a natural order on $\mathcal{X}_T$, in the sense that for any given $w$ and $\tilde{w}$ in $\mathcal{X}_T$, we say that
	$$ 
	w\le \tilde{w}\quad \text{if and only if}\quad  \tilde{w}-w\in \mathcal{X}_T^+.
	$$
	Moreover we have that $w<\tilde{w}$ (resp. $w\ll \tilde{w}$) if $w\le \tilde{w}$ and $\tilde{w}-w\neq 0$ (resp. $\tilde{w}-w\in\mathcal{X}_T^{++}$). We extend the order on $\mathcal{X}_T$ componentwise on the product space $\mathcal{X}_T\times\mathcal{X}_T$, that is, given ${\bf w}=(w_1,w_2)$ and $\tilde{\bf w}=(\tilde{w}_1,\tilde{w}_2)\in \mathcal{X}_T\times\mathcal{X}_T$, we say that ${\bf w}\le \tilde{\bf w}$ (resp. ${\bf w}<\tilde{\bf w}$, ${\bf w}\ll \tilde{\bf w}$) if $\tilde{w}_i\le w_i$ (resp. $\tilde{w}_i<w_i$, $\tilde{w}_i\ll w_i$ ) for each $i=1,2$. We shall endow $C(\overline{\Omega})\times C(\overline{\Omega})$ with the max-norm
	$$ 
	\|{\bf u}\|:=\max\{\|u_1\|_{\infty},\|u_2\|_{\infty}\}\quad \forall \ {\bf u}\in C(\overline{\Omega})\times C(\overline{\Omega}).
	$$
	Similarly, on $\mathcal{X}_T\times\mathcal{X}_T$, we   consider the norm 
	\begin{equation}
		\|{\bf w}\|_T:=\max\{\|w_1\|_{\mathcal{X}_T},\|w_2\|_{\mathcal{X}_T}\}\quad \forall\ {\bf w}\in\mathcal{X}_T\times\mathcal{X}_T.
	\end{equation}
	For every $p>1$, set  $$
	{\rm Dom}_p(\Delta):=\{u\in W^{2,p}(\Omega)\ :\ \partial_{\vec{n}}u=0 \ \text{on}\ \partial\Omega\}$$ 
	{and} 
	$${\rm Dom}_{\infty}(\Delta)=\{u\in\cap_{p\ge 1}W^{2,p}(\Omega)\ :\ \Delta u\in C(\overline{\Omega})\}.
	$$ 
	It is well known that the Laplace operator $\Delta$ generates an analytic  $c_0$-semigroup $\{e^{t\Delta}\}_{t\ge 0}$ on $L^p(\Omega)$ with  domain ${\rm Dom}_q(\Delta)$. It is also known that $\Delta$ generates an analytic $c_0$-semigroup on $C(\bar{\Omega})$ with domain ${\rm Dom}_{\infty}(\Delta)$ (\cite{Amann1983}). 
	
	\medskip

	Set $\mathcal{D}:= [{\rm Dom}_{\infty}(\Delta)]^2$ and  define the mapping $\mathcal{A}  : \mathbb{R}\times\mathcal{D}  \to C(\bar{\Omega})\times C(\bar{\Omega})$ by 
	\begin{align*}
		\mathcal{A}(t){\bf u}  =&\left(\begin{array}{c} 
			d_{1}\Delta u_{1} + r(t,\cdot)u_{2} - (s(t,\cdot)+a(t,\cdot))u_{1}	\\
			d_{2}\Delta u_{2} + s(t,\cdot)u_{1} - e(t,\cdot)u_{2} 
		\end{array} \right)\quad\ t\in\mathbb{R},\  {\bf u} \in \mathcal{D}.
	\end{align*} 
	Hence, \eqref{Eq1} is equivalent to
	\begin{equation}\label{Eq0-1} 
		\begin{cases}
			\partial_{t}{\bf u}(t,\cdot) = \mathcal{A}(t){\bf u}(t,\cdot) + G(t,{\bf u}(t,\cdot)) & x \in \Omega,  t >t_0 \\ 
			\partial_{\vec{n}}{\bf u}(t,\cdot) = 0 & x \in \partial\Omega, t >t_0 \cr 
			{\bf u}(t,\cdot)={\bf u}_0 & x\in\overline{\Omega},\ t=t_0,
		\end{cases} 
	\end{equation}
	where 
	\begin{equation}\label{GG-1}
		G(t,{\bf u}) = \left(\begin{array}{c} 
			-(b(t,.)u_{1} + c(t,.)u_{2})u_{1}	\\
			-( f(t,.)u_{2} + g(t,.)u_{1})u_{2} 
		\end{array} \right)\quad {\bf u}\in C(\overline{\Omega})\times C(\overline{\Omega}), \ t\in\mathbb{R}.
	\end{equation}
	
	Given a bounded function $h : \mathcal{S}\subset\Omega\times\mathbb{R}\to \mathbb{R}$, we set 
	$$
	h_{\inf}:=\inf_{(t,x)\in \mathcal{S}}h(t,x)\quad \text{and}\quad h_{\sup}=\sup_{(t,x)\in\mathcal{S} }h(t,x).
	$$
	When the infimum (resp. supremum) is attained, we shall use $h_{\min}$ (resp. $h_{\max}$) for $h_{\inf}$ (resp. $h_{\sup}$). Throughout this work, we suppose that the following standing hypothesis holds:
	
	\medskip 
	
	\noindent{\bf (H1)} The model parameter functions $a$, $b$, $c$, $e$, $f$, $g$, $r$, and $s$ are  H\"older continuous in both variables, time periodic with a common period $T>0$, and  $\min\{r_{\min},s_{\min},b_{\min},f_{\min}\}>0$. 
	\medskip
	
	\noindent The H\"older continuity of the model parameters specified in {\bf (H1)} ensures that solutions of \eqref{Eq1} exhibit regularity in the classical sense. Additionally, strict positivity of the intra-stage self-limitation coefficients $b$ and $f$ is imposed for ensuring the global boundedness of solutions (see inequality \eqref{NN-Eq1}). The condition $\min\{r_{\min}, s_{\min}\} > 0$ indicates that all spatial locations consistently support species growth and reproduction. It is worth noting that the latter condition can be relaxed, although we imposed it for technical reasons in the proofs of our main results. 
	
	\medskip
	
	It follows from {\bf (H1)} and \cite[Theorem 7.1.3]{Henry} that $\mathcal{A}(t)$ generates an evolution $c_0$-semigroup $\{{\bf U}(t,s)\}_{t\ge s}$, in the sense that, given an initial data ${\bf u}_0\in C(\overline{\Omega})\times C(\overline{\Omega})$, then  ${\bf v}(t,x,s):=({\bf U}(t,s){\bf u}_0)(x)$, $t\ge s$ is the unique classical solution of 
	\begin{equation}\label{Eq0-2}
		\begin{cases}
			\partial_t{\bf v}=\mathcal{A}(t){\bf v} & x\in\Omega, \ t>s,\cr 
			0=\partial_{\vec{n}}{\bf v} & x\in\partial\Omega,\  t>s,\cr
			{\bf v}(t,\cdot,s)={\bf u}_0 & x\in\Omega,\ t=s.
		\end{cases}
	\end{equation}
	Furthermore, thanks to  {\bf (H1)}, the function $(t,{\bf u})\mapsto G(t,{\bf u})$ is locally Lipschitz in ${\bf u}$ uniformly in $t\in\mathbb{R}$, and  H\"older continuous in $t$  uniformly for   ${\bf u}$ on bounded sets. Thus,  by \cite[Theoreme 3.3.4]{Henry}, for any  ${\bf u}_0\in C(\overline{\Omega})\times C(\overline{\Omega})$ and initial data $t_0\in\mathbb{R}$,  system \eqref{Eq1} has a unique classical solution  ${\bf u}(t,\cdot;{\bf u}_0,t_0)$ defined on a maximal interval of existence  $[t_0, t_0+T_{\max})$ for some $T_{\max}\in(0,\infty]$. Moreover, since $\mathcal{A}(t)$ is cooperative, if ${\bf u}_0\in [C^+(\bar{\Omega})]^2$, it follows from the comparison principle for parabolic equations that ${\bf u}(t,\cdot;{\bf u}_0,t_0)\in[C^+(\bar{\Omega})]^2 $  for every $t\in [t_0,t_0+T_{\max})$. Therefore, since $G(t,{\bf u}(t,\cdot;{\bf u}_0))\le {\bf 0}$ whenever ${\bf u}_0\in [C^+(\bar{\Omega})]^2$, then ${\bf u}(t,\cdot;{\bf u}_0,t_0)\le {\bf U}(t,t_0){\bf u}_0$ for any $t\in[t_0,t_0+t_{\max})$. Hence,  $T_{\max}=\infty$. Furthermore, due to the type of nonlinearity in \eqref{Eq0-1}, we have that ${\bf u}(t,\cdot;{\bf u}_0,t_0)$ is uniformly bounded in $t\ge t_0$ (see inequality \eqref{NN-Eq1} ). 
	
	\medskip
	
	We linearize \eqref{Eq1} at the trivial solution ${\bf 0}$ and consider the associated eigenvalue problem
	
	\begin{equation}\label{Eq2}
		\begin{cases}
			\lambda\varphi=\mathcal{A}(t)\varphi -\partial_t\varphi & x\in\Omega,\ t>0,\cr 
			0=\partial_{\vec{n}}\varphi & x\in\partial\Omega,\ t>0,\cr
			\varphi(0,\cdot)=\varphi(T,\cdot).
		\end{cases}
	\end{equation}
	Since the operator at the right hand-side of system \eqref{Eq2} is cooperative, it follows from the Krein-Rutman theorem that the periodic eigenvalue problem \eqref{Eq2} has a principal eigenvalue, say $\lambda_{*}$, with a corresponding positive eigenfunction ${\bf \varphi}=(\varphi_1,\varphi_2)\in\mathcal{X}_T^{++}\times\mathcal{X}_T^{++}$. We normalize the eigenvalue, so that $\|{\bf \varphi}\|_{T}=1$. Note that the function ${\bf \Phi}(t,x):=e^{\lambda_{*}t}{\bf \varphi}(t,x)$ solves the linear cooperative system \eqref{Eq0-2}.
	
	\subsection{Main Results}
	
	Our first result concerns the extinction of the population and reads as follows.
	
	\begin{tm}\label{TH1} Assume {\bf (H1)} and  $\lambda_{*}\le 0$. Then every classical solution ${\bf u}(t,\cdot)$ of \eqref{Eq1} with a positive initial data satisfies 
		\begin{equation}\label{TH1-eq1}
			\lim_{t\to\infty}\sup_{t_0\in\mathbb{R}}\|{\bf u}(t+t_0,\cdot;{\bf u}_0,t_0)\|=0.
		\end{equation}
	\end{tm}
	
	Thanks to Theorem \ref{TH1}, we see that the population will eventually go extinct if $\lambda_*\le 0$. A natural question is to inquire what happens when $\lambda_*>0$. Our next result is concerned with the uniform persistence of the population when $\lambda_*>0$.
	
	\begin{tm}\label{TH2}
		Assume {\bf (H1)} and $\lambda_{*}> 0$ and fix a positive initial data ${\bf u}_0\in[C^+(\bar{\Omega})]^2\setminus\{\bf 0\}$. Then there is $\eta_*>0$ such that  
		\begin{equation}\label{TH2-eq1}
			\eta_*{\bf 1}\le {\bf u}(t+t_0,\cdot;{\bf u}_0,t_0)\quad \forall\ t\ge T, \ t_0\in\mathbb{R}.
		\end{equation}
		
	\end{tm}
	
	Theorem \ref{TH2} shows that the population persists if $\lambda_*>0$. Thanks to the last two results, it is of great importance to know the underlying hypotheses on the model parameters which yield $\lambda_*>0$. However, it is known that $\lambda_*$ may depend nontrivially on the diffusion rates $d_1$ and $d_2$. For example, the work \cite{CCS2023} studied how the diffusion rates $d_1$ and $d_2$ affect $\lambda_*$ when the parameters are temporally homogeneous. In the latter case, the results of \cite{CCS2023} showed that $\lambda_*$ may not depend monotonically on the diffusion rates which is in strong contrast to the case of single-stage population models.  Our next result provides sufficient assumptions on the model parameters which give the sign of $\lambda_*$ irrespective of the population's diffusion rates.
	
	\begin{prop}\label{Prop1} Assume {\bf (H1)} and let $\lambda_*$ be the principal eigenvalue of \eqref{Eq2}. Then the following conclusions hold.
		\begin{itemize}
			\item[\rm (i)] If  either  $(r+s)^2\le,\ne 4(a+s)e$  or $\big(\frac{a+s}{r}\big)_{\min}\big(\frac{e}{s}\big)_{\min}>1$, then $\lambda_*<0$.
			\item[\rm (ii)] 
			If either $\int_{0}^T\int_{\Omega}\sqrt{rs}>\frac{1}{2}\int_0^T\int_{\Omega}(a+s+e)$ or  $\big(\frac{a+s}{r}\big)_{\max}\big(\frac{e}{s}\big)_{\max}<1$, then $\lambda_*>0$. 
		\end{itemize}
		
	\end{prop}
	
	\medskip
	
	Next, when the species persists, it is usually the case for most unstructured population models that there exists some  positive entire solution. By a positive entire solution of \eqref{Eq1}, we mean a positive classical  solution of \eqref{Eq1} defined for all time $t\in\mathbb{R}$. A similar result also holds for the time-periodic two-stage structured model \eqref{Eq1} as established in our next result. 
	
	\begin{tm}\label{TH3} Assume {\bf (H1)} and $\lambda_*>0$. 
		\begin{itemize}
			\item[\rm (i)] System \eqref{Eq1} has at least one  $T$-periodic strictly positive entire solution.
			\item[\rm (ii)] If the coefficients are time-independent, then system \eqref{Eq1} has at least one positive steady state solution.
			\item[\rm (iii)] If all the parameters are constant, then system \eqref{Eq1} has a constant positive equilibrium solution.
		\end{itemize}
		
	\end{tm}
	The question of the uniqueness and stability of the entire positive solution of \eqref{Eq1} turns out to be  challenging. In spatially and temporally homogeneous environments, a positive constant steady-state solution of \eqref{Eq1} is unique whenever it exists. Moreover, the linearized system at this positive equilibrium solution is always cooperative. This leads to the question of whether any entire solution of the nonhomogeneous system \eqref{Eq1} is unique and globally stable when the linearization of \eqref{Eq1} at such a solution is cooperative. Remarkably, our subsequent result provides an affirmative answer to this question.  To state this result, we introduce the following assumption.
	
	\medskip
	
	{\bf (H2)} $\lambda_*>0$ and \eqref{Eq1} has a strictly positive T-periodic  entire solution $\tilde{\bf u}(t,x)$ satisfying $r(t,x)\ge c(t,x)\tilde{u}_1(t,x)$ and $s(t,x)\ge g(t,x)\tilde{u}_2(t,x)$ for every $x\in\bar{\Omega}$ and $t\in\mathbb{R}$. 
	
	\medskip
	
	Note that, since  $0<\tilde{u}_1\le \frac{\max{\{\|r\|_{\infty},\|s\|_{\infty}\}}}{\min\{b_{\min},f_{\min}\}}$ and $0<\tilde{u}_2\le \frac{\max{\{\|r\|_{\infty},\|s\|_{\infty}\}}}{\min\{b_{\min},f_{\min}\}}$ for any positive entire solution $\tilde{\bf u}(t,x)$ of \eqref{Eq1},  {\bf (H2)} holds for small $\|c\|_T$ and $\|g\|_{T}$ whenever $\lambda_*>0$. Note also that if $\tilde{\bf u}(t,\cdot)$ is a strictly positive $T$-time periodic solution of \eqref{Eq1} satisfying hypothesis {\bf (H2)}, the linearization of \eqref{Eq1} at $\tilde{\bf u}$ is cooperative. In which case, it can be shown that $\tilde{\bf u}(t,\cdot)$ is linearly stable. Our next result shows that the following stronger result holds.
	
	\begin{tm}\label{TH4} Assume {\bf  (H1)} and suppose that \eqref{Eq1}  has a positive T-periodic solution $\tilde{\bf u}(t,x)$ satisfying {\bf (H2)}. Then, for every  classical solution ${\bf u}(\cdot,\cdot;{\bf u}_0,t_0)$ of \eqref{Eq1} with a positive initial data ${\bf u}_0$ and initial time $t_0\in\mathbb{R}$, it holds that  
		\begin{equation} \label{TH4-eq1} 
			\lim_{t\to\infty}\sup_{t_0\in\mathbb{R}}\|{\bf u}(t+t_0,\cdot;{\bf u}_0,t_0)-\tilde{\bf u}(t+t_0,\cdot)\|=0. 
		\end{equation}
		
	\end{tm}
	Theorem \ref{TH4} provides a sufficient condition for the uniqueness and stability of strictly positive $T$-periodic and bounded solution of \eqref{Eq1}. For practical applications, it would be helpful to find some way to verify whether hypothesis {\bf (H2)} holds. Our next result goes along this direction and provides sufficient conditions to construct examples for which hypothesis {\bf (H2)} holds.

	\begin{prop}\label{Cor2} Assume that {\bf (H1)} holds, $\lambda_*>0$, and  $(\tilde{r}, \tilde{s}) := (r/c, s/g)$  satisfies  
		\begin{equation}\label{Rk1-eq1}
			\begin{cases}
				\partial_t\tilde{r}\ge d_1\Delta\tilde{r}-(a+{s}+b\tilde{r})\tilde{r} & x\in\Omega,\ t\in\mathbb{R},\cr 
				\partial_{t}\tilde{s}\ge d_2\Delta\tilde{s}-(e+f\tilde{s})\tilde{s} & x\in\Omega,\ t\in\mathbb{R},\cr 
				\partial_{\vec{n}}\tilde{r}\ge 0,\ \partial_{\vec{n}}\tilde{s}\ge 0 & x\in\partial\Omega,\ t\in\mathbb{R}.
			\end{cases}
		\end{equation}
		Then system \eqref{Eq1} has a unique strictly positive entire $T$-periodic $\tilde{\bf u}$ solution satisfying {\bf (H2)}.  Furthermore, $\tilde{\bf u}$ is uniformly globally stable with respect to solutions with positive initial  in the sense of \eqref{TH4-eq1}. 
		
	\end{prop}
	
	The next result is a consequence  of Theorem \ref{TH4} and Proposition \ref{Cor2}.
	
	\begin{coro}\label{Cor1} Assume that {\bf (H1)} holds,    $\lambda_*>0$, and  the functions $c/r$ and $g/s$ are constant. Then {\bf (H2)} holds and  \eqref{Eq1} has a unique  time $T$-periodic positive  solution $\tilde{\bf u}$. Furthermore, $\tilde{\bf u}$  is globally asymptotically stable with respect to all nonnegative and nontrivial solutions in the sense of \eqref{TH4-eq1}.     
	\end{coro}
	\begin{rk} Thanks to Proposition \ref{Prop1}-{\rm (ii)} and Corollary \ref{Cor1}, if $\big(\frac{a+s}{r}\big)_{\max}\big(\frac{e}{s}\big)_{\max}<1$ and the functions $c/r$ and $g/s$ are constant, then every classical solution of \eqref{Eq1} with a positive initial data eventually converges to the unique $T$-periodic positive entire solution.
	\end{rk}
	\begin{rk} Thanks to Corollary \ref{Cor1}, when all the parameters are temporally and spatially constant, then every solution of \eqref{Eq1} eventually stabilizes either  (i)  at the trivial solution if $\lambda_*\le 0$, or (ii) at  the unique positive constant equilibrium solution if $\lambda_*>0$. In this case, we have that $\lambda_*$ is independent of population diffusion rates and the global dynamics of solution of the diffusive population model \eqref{Eq1} is determined by that of the corresponding kinetic system.
	\end{rk}
	
	\begin{exa} Thanks to Proposition \eqref{Cor2},  we present a few examples of parameters such that {\bf (H2)} holds whenever $\lambda_*>0$. First fix $a, b, e, f, c$ and $g$  as in {\bf (H1)}. Second, fix $\tilde{s}_0\in \mathcal{X}_T^+$ such that   the periodic logistic reaction diffusion equation
		\begin{equation}
			\begin{cases}
				\partial_t\tilde{s}=d_2\Delta \tilde{s} +(\tilde{s}_0-e-f \tilde{s}) \tilde{s} & x\in\Omega, 0<t<T\cr 
				0=\partial_{\vec{n}}\tilde{s} & x\in\partial\Omega,\ 0<t<T,\cr
				\tilde{s}(0,\cdot)=\tilde{s}(T,\cdot),
			\end{cases}
		\end{equation}
		has a   unique $T$-periodic positive entire solution, denoted as  $\tilde{s}^*(t,x)$.   Third, fix $\tilde{r}_0\in \mathcal{X}_T^+$ such that  the periodic logistic reaction diffusion equation 
		\begin{equation}
			\begin{cases}
				\partial_t\tilde{r}=d_1\Delta \tilde{r} +(\tilde{r}_0-a-b \tilde{r}) \tilde{r} & x\in\Omega,\ 0<t<T\cr 
				0=\partial_{\vec{n}}\tilde{r} & x\in\partial\Omega, \ 0<t<T,\cr
				\tilde{r}(0,\cdot)=\tilde{r}(T,\cdot),
			\end{cases}
		\end{equation}
		has a unique $T$-periodic positive entire solution, denoted as  $\tilde{r}^*(t,x)$. Then for $r=c\tilde{r}^*$ and $s=g\tilde{s}^*$ in \eqref{Eq1},  {\bf (H2)} holds whenever $\lambda_*>0$.    
	\end{exa}
	
	It is clear that when all the parameters are positive constants, then Proposition \ref{Cor2} is trivially applicable. It turns out that when the diffusion rates $d_1$ and $d_2$ are small, hypothesis {\bf (H2)} is also necessary. To be precise, we need the following hypothesis.
	
	\medskip
	
	{\bf (H3)} All the parameter functions are time independent, H\"older continuous on $\bar{\Omega}$, $e_{\min}>0$ and $(\frac{rs}{a+s}-e)_{\max}>0$.
	
	\medskip When  the environment is time homogeneous, it follows from \cite{CCM2020} that system \eqref{Eq1} has at least one positive steady-state  for small  diffusion rates of the population if $(\frac{rs}{a+s}-e)_{\max}>0$. However,  positive classical solutions of \eqref{Eq1} eventually go extinct for small diffusion rates of the population if   $(\frac{rs}{a+s}-e)_{\max}<0$. Hence, the fact that $(\frac{rs}{a+s}-e)_{\max}>0$ in {\bf (H3)} ensures the existence of positive steady-states of \eqref{Eq1} for small diffusion rates. Under hypotheses {\bf (H1)} and {\bf (H3)}, we have the following result on the uniqueness and global stability of positive steady state solutions of \eqref{Eq1}.

	\begin{tm}\label{TH5} Suppose that {\bf (H1)} and {\bf (H3)} hold. Then there is $d_0>0$ such that for every $0<d_1,d_2<d_0$, system \eqref{Eq1} has unique positive steady state solution ${\bf u}(\cdot;{\bf d})$. Furthermore, ${\bf u}(\cdot;{\bf d})$ is globally stable in the sense that for every classical solution ${\bf u}(t,x;{\bf u}_0)$ of \eqref{Eq1} with a positive initial data, it holds that 
		\begin{equation}
			\lim_{t\to\infty}\|{\bf u}(t,\cdot;{\bf u}_0)-{\bf u}(\cdot;{\bf d})\|=0.
		\end{equation}
		
	\end{tm}
	
	Our approach to establish Theorem \ref{TH5} is to show that under hypotheses {\bf (H1)} and {\bf (H3)}, hypothesis {\bf (H2)} holds for the range of diffusion rates as selected in the statement  of the theorem. Hence, we see that Theorem \ref{TH5} is also a consequence of Theorem \ref{TH4}. Moreover,  under hypotheses {\bf (H1)} and {\bf (H3)}, Theorem \ref{TH5} indicates  that if both subpopulations move slowly, then the population will always stabilize at the unique positive steady state. An immediate question raised by Theorem \ref{TH5} is whether  the conclusions of the theorem still hold when the parameters are time dependent. We plan to further explore these questions in our future works.
	
	\medskip
	
	We complement Theorem \ref{TH5} with the following result on the spatial distributions of the population when the diffusion rates are small.
	
	\begin{tm}\label{TH6} Suppose that {\bf (H1)} and {\bf (H3)} hold. Let $d_0>0$ be as in Theorem \ref{TH5}. Then for every $0<d_1,d_2<d_0$, the unique positive steady state solution ${\bf u}(\cdot;{\bf d})$ of \eqref{Eq1} satisfies 
		\begin{equation}\label{TH6-eq1}
			\lim_{\max\{d_1,d_2\}\to 0}{\bf u}(x;{\bf d})={\bf u}^*(x) \quad \text{for}\ x\ \text{locally uniformly in}\ \Omega,
		\end{equation}
		where for each $x\in\bar{\Omega}$, ${\bf u}^*(x)=(u_1^*(x),u_2^*(x))$ is the unique nonnegative stable solution of the system of algebraic equations
		\begin{equation}\label{TH6-eq2}
			\begin{cases}
				0=r(x)u_2^*-(a(x)+s(x)+b(x)u_1^*+c(x)u_2^*)u_1^*\cr 
				0=s(x)u_1^*-(e(x)+f(x)u_2^*+g(x)u_1^*)u_2^*.
			\end{cases}
		\end{equation}
		
	\end{tm}
	
	Theorem \ref{TH6} establishes the asymptotic profiles of the steady state solutions of \eqref{Eq1} when the diffusion rates are small and the environments is temporally constant. This result improves the previous results in \cite[Theorem 1]{CCM2020} and \cite[Theorem 2.4]{CCS2023}, where Theorem \ref{TH6} was established under some smallest assumptions on $c$ and $g$. Here we require no restriction on the functions $c$ and $g$.  Theorem \ref{TH6} indicates that when both the juveniles and adults diffuse slowly, then the population will concentrate on the sites where the product of the maturity and reproduction rates $r(x)s(x)$ exceeds the product of the local   cumulative lost rates $(a(x)+s(x))e(x)$. Hence, when the population diffuses slowly, the intra-stage and inter-stage competition rates play no role on its survival. This is similar to what it is known for unstructured population models of logistic-type  reaction terms (\cite[Proposition 3.16]{CC2003}). It would be interesting to know the extend to which the conclusion of Theorem \ref{TH6} holds when the environment is both temporally and spatially heterogeneous.
	
	\section{Proofs of the Main Results}\label{sec3}
	Throughout the rest of this work, we shall always suppose that {\bf (H1)} holds. We collect a few preliminary results. First, set 
	$$M:=\frac{ \varphi_{\max}}{ \varphi_{\min}}\quad  \text{where}\quad  \varphi_{\max}=\max_{x\in\overline{\Omega},t\in[0,T], i=1,2}\varphi_i(t,x)\quad \text{and}\quad   \varphi_{\min}:=\min_{x\in\overline{\Omega},t\in[0,T],i=1,2}\varphi_{i}(t,x).
	$$
	Let $\{{\bf U}(t,s)\}_{t>s}$ be the evolution operator induced by solutions of \eqref{Eq0-2}. Then,
	\begin{equation}\label{Ra1}
		\|{\bf U}(t,s)\|\le Me^{\lambda_{*}(t-s)}\quad \forall\ t\ge s.
	\end{equation}
	Indeed, let ${\bf u}\in [C(\bar{\Omega})]^2$. Observing that
	$$
	-\frac{\|{\bf u}\|}{\varphi_{\min}}{ \varphi}(s,\cdot)\le {\bf u}\le \frac{\|{\bf u}\|}{\varphi_{\min}}{ \varphi}(s,\cdot)\quad \forall\ s\in\mathbb{R},
	$$
	then,  since ${\bf U}(t+s,s)$ is positive,
	\begin{align*}
		{\bf U}(t+s,s)\Big(-\frac{\|{\bf u}\|}{\varphi_{\min}}{ \varphi}(s,\cdot)\Big)\le {\bf U}(t+s,s){\bf u} 
		\le {\bf U}(t+s,s)\Big(\frac{\|{\bf u}\|}{\varphi_{\min}}{ \varphi}(s,\cdot)\Big)\quad \forall\ t\ge 0,\ s\in\mathbb{R}.
	\end{align*}
	Recalling that ${\bf U}(t+s,s)({ \varphi}(s,\cdot))=e^{\lambda_{*}t}{ \varphi}(t+s,\cdot)$, we deduce that
	\begin{equation}\label{Eq0-6}
		-\frac{\|{\bf u}\|}{\varphi_{\min}}e^{\lambda_{*}t}{ \varphi}(t+s,\cdot)\le {\bf U}(t+s,s){\bf u}\le \frac{\|{\bf u}\|}{\varphi_{\min}}e^{\lambda_{*}t}{ \varphi}(t+s,\cdot)\quad \forall\ t\ge 0, \ s\in\mathbb{R}. 
	\end{equation}
	Since ${\bf u}\in [C(\overline{\Omega})]^2$ is arbitrary, then
	\begin{equation*}
		\|{\bf U}(t+s,s)\|\le \frac{\varphi_{\max}}{\varphi_{\min}}e^{\lambda_{*}t}\quad t\ge 0,\ s\in\mathbb{R},
	\end{equation*}
	which yields  \eqref{Ra1}. Now, by  \eqref{Ra1} and the fact ${\bf 0}\le {\bf u}(t+t_0,\cdot;{\bf u}_0,t_0)\le {\bf U}(t+t_0,t_0){\bf u}_0$, $t\ge 0$, $t_0\in\mathbb{R}$, ${\bf u}_0\in[C^+(\bar{\Omega})]^2$, we have  
	
	\begin{equation}\label{Eq0-5}
		\| {\bf u}(t+t_0,\cdot,{\bf u}_0,t_0)\|\le Me^{t\lambda_{*}}\|{\bf u}_0\|\quad t\ge 0,\ t_0\in\mathbb{R},\quad \text{and}\ {\bf u}_0\in[C^+(\bar{\Omega})]^2.
	\end{equation}
	
	\medskip
	
	It is clear from \eqref{Eq0-5} that solutions of the initial value problem \eqref{Eq1} are globally bounded in time if $\lambda_*\le 0$. To see that this global boundedness also holds when $\lambda_*>0$, we note that, given ${\bf u}_0\in [C^+(\bar{\Omega})]^2$, then for every $t_0\in\mathbb{R}$, ${\bf u}(t+t_0,\cdot;{\bf u}_0,t_0)$ satisfies 
	\begin{equation*}
		\begin{cases}
			\partial_tu_1\le d_1\Delta u_1 +\|r\|_{\infty}u_2 -b_{\min}u_1^2 & x\in\Omega,\ t>0,\cr 
			\partial_tu_2\le d_2\Delta u_2 +\|s\|_{\infty}u_1-f_{\min}u_2^2 & x\in\Omega,\ t>0,\cr
			0=\partial_{\vec{n}}{\bf u} &x\in\partial\Omega,\ t>0.
		\end{cases}
	\end{equation*}
	Thus, observing that the constant function $\bar{\bf u}(t,x)=M{\bf 1}$ for all $t\ge 0$, where $M=\max\{\|{\bf u}_0\|, \frac{\max\{\|r\|_{\infty},\|s\|_{\infty}\}}{\min\{b_{\min},f_{\min}\}}\}$ satisfies
	\begin{equation*}
		\begin{cases}
			\partial_t\bar{u}_1\ge d_1\Delta \bar{u}_1 +\|r\|_{\infty}\bar{u}_2 -b_{\min}\bar{u}_1^2 & x\in\Omega,\ t>0,\cr 
			\partial_t\bar{u}_2\ge d_2\Delta \bar{u}_2 +\|s\|_{\infty}\bar{u}_1-f_{\min}\bar{u}_2^2 & x\in\Omega,\ t>0,\cr
			0=\partial_{\vec{n}}\bar{\bf u} &x\in\partial\Omega,\ t>0,
		\end{cases}
	\end{equation*}
	and ${\bf u}_0\le \bar{\bf u}$, we can employ the comparison principle for parabolic cooperative systems to conclude that ${\bf u}(t+t_0,\cdot;{\bf u}_0,t_0)\le \bar{\bf u}(t,\cdot)$ for all $t\ge 0$, from which it follows that 
	\begin{equation}\label{NN-Eq1}
		\|{\bf u}(t+t_0,\cdot;{\bf u}_0,t_0)\|\le \max\Big\{\|{\bf u}_0\|, \frac{\max\{\|r\|_{\infty},\|s\|_{\infty}\}}{\min\{b_{\min},f_{\min}\}}\Big\}\quad \forall\ t_0\in\mathbb{R} \ \text{and}\ t\ge 0.
	\end{equation}

	\subsection{Proof of Theorem \ref{TH1}}

	Note from \eqref{Eq0-5} that ${\bf u}(t+t_0,\cdot,{\bf u}_0,t_0)$ converges exponentially  to zero uniformly in $t_0\in\mathbb{R}$ if $\lambda_{*}<0$. To handle the case of $\lambda_*=0$, we start with the following lemma.
	\begin{lem}\label{lem1}
		Let $\lambda_{*}$ be the principal eigenvalue  of \eqref{Eq2} and $\varphi$ be the corresponding strictly positive eigenfunction satisfying $\|\varphi\|_T = 1$. Let $ {\bf u}(t,\cdot;{\bf u}_0,t_0)$ be a classical solution of  \eqref{Eq1} with a positive initial data ${\bf u}_{0}\in[C^+(\bar{\Omega})]^2$ and initial time $t_0\in\mathbb{R}$.  For every $t\ge 0$, define
		\begin{equation}\label{def-c}
			\sigma(t,{\bf u}_{0}) = \inf\{ \sigma >0 :  {\bf u}(t+t_0,\cdot;{{\bf u}_0,t_0})\le \sigma e^{\lambda_{*} t}{ \varphi}(t+t_0,\cdot),\  \forall\ t_0\in\mathbb{R}\}.
		\end{equation}
		Then, ${\sigma}(t,{\bf u}_0)\in (0,\infty)$ for every $t\ge 0$ and the function $[0,\infty)\ni t\mapsto \sigma(t,{\bf u}_0)$ is nonincreasing.
	\end{lem}
	\begin{proof}  Fix ${\bf u}_0\in [C^+(\bar{\Omega})]^2\setminus\{\bf 0\}$ and we proceed in two steps.
		
		\medskip
		
		\noindent {\bf Step 1.} As ${\bf 0}\le {\bf u}(t+t_0,\cdot;{\bf u}_0,t_0)\le {\bf U}(t+t_0,t_0){\bf u}_0$, for every  $t\ge 0$ and   $t_0\in\mathbb{R}$, it follows from   \eqref{Eq0-6} that  $\sigma(t,{\bf u}_0)$ is well defined for all $t\ge 0$.  Note again from \eqref{Eq0-6} that 
		$$
		{\bf u}(t+t_0,\cdot;{\bf u}_0,t_0)\le \frac{\|{\bf u}_0\|}{\varphi_{\min}}e^{t\lambda_{*}}{ \varphi}(t+t_0)\quad \forall\ t\ge 0, \ t_0\in\mathbb{R}.
		$$
		Thus, $0<\sigma(t,{\bf u}_0)\le \|{\bf u}_0\|/\varphi_{\min}$ for every $t\ge 0$.

		\medskip
		
		\noindent{\bf Step 2.}  We show that $\sigma(t,{\bf u}_{0}) \leq \sigma(\tau,{\bf u}_{0})$ whenever $t\ge \tau\ge 0$. To this end, fix $\tau\ge 0$ and  $t_0\in\mathbb{R}$, and set 
		$$
		\tilde{\bf v}(t):= \sigma(\tau,{\bf u}_{0})e^{\lambda_{*}(t+\tau)} { \varphi}(t+t_{0}+\tau,\cdot)\quad \forall\ t>0.
		$$
		Recall that $\Phi(t,\cdot)=e^{t\lambda_{*}}{ \varphi}(t,\cdot)$ satisfies \eqref{Eq0-2}. Hence 
		\begin{equation}\label{g5}
			\partial_t\Phi(t+t_{0}+\tau,\cdot)=\mathcal{A}(t+t_0+\tau,\cdot )\Phi(t+t_0+\tau,\cdot)  \quad \forall\ t\ge 0. 
		\end{equation} 
		Multiplying this equation by $\sigma(\tau,{\bf u}_0)e^{-\lambda_{*}t_0}$, we get
		
		$$
		\partial_t(\sigma(\tau,{\bf u}_0)e^{-\lambda_{*}t_0}\Phi(t+t_{0}+\tau,\cdot))=\mathcal{A}(t+t_0+\tau,\cdot )(c(\tau,{\bf u}_0)e^{-\lambda_{*}t_0}\Phi(t+t_0+\tau,\cdot)) 
		$$ 
		because  $\sigma(\tau,{\bf u}_{0})e^{-t_0\lambda_{*}}$ is constant and independent of $t$. Observing that  
		\begin{align*}
			\tilde{\bf v}(t,\cdot)=\sigma(\tau,{\bf u}_0)e^{\lambda_{*}(t+\tau)}{ \varphi}(t+\tau+t_0,\cdot)= \sigma(\tau,{\bf u}_0,t_0)e^{-\lambda_{*}t_0}\Phi(t+t_0+\tau,\cdot)\,\quad  t\ge 0,
		\end{align*}
		so  $\tilde{\bf v}(t,\cdot)$ satisfies 
		\begin{align}\label{R-eq1}
			\begin{cases}
				\partial_{t}\tilde{\bf v}(t,\cdot)=\mathcal{A}(t+\tau+t_0,\cdot)\tilde{\bf v}(t,\cdot) & x\in\Omega,\ t>0,\cr
				0=\partial_{\vec{n}}\tilde{\bf v}(t,\cdot) & x\in\partial\Omega, t>0\cr
				\tilde{\bf v}(0,\cdot)=\sigma(\tau,{\bf u}_0,t_0)e^{-\lambda_{*}\tau}{ \varphi}(\tau+t_0,\cdot) & t=0, x\in\overline{\Omega}.
			\end{cases}
		\end{align}
		On the other hand, setting $\hat{\bf v}(t,\cdot)= {\bf u}(t+t_0+\tau,\cdot;{\bf u}_0,t_0)$, $t\ge 0$, we obtain  
		\begin{equation}\label{R-eq2}
			\begin{cases}
				\partial_t\hat{\bf v}(t,\cdot)\le \mathcal{A}(t+t_0+
				\tau,\cdot)\hat{\bf v}(t,\cdot) & x\in\Omega, t>0\cr 
				\partial_{\vec{n}}\hat{\bf v}(t,\cdot) = 0 & x\in\partial\Omega, t >0 \cr 
				\hat{\bf v}(0,\cdot) = {\bf u}(t_{0}+\tau,\cdot,t_0) & x\in\overline{\Omega}, t=0,
			\end{cases} 
		\end{equation}
		because $G(t,{\bf u}(t,\cdot;{\bf u}_0,t_0))\le {\bf 0}$, where ${G(t,{\bf u})}$  is given by \eqref{GG-1}. 
		Recalling from \eqref{def-c} that 
		$$\tilde{\bf v}(0,\cdot)=c(\tau,{\bf u}_0)e^{\lambda_{*}\tau}{ \varphi}(t_0+\tau)\ge {\bf u}(\tau+t_0,\cdot,t_0)=\hat{\bf v}(0,\cdot),
		$$ 
		we deduce from \eqref{R-eq1}, \eqref{R-eq2}, and the comparison principle for systems of cooperative parabolic equations that   
		$$
		\hat{\bf v}(t) \leq \tilde{\bf v}(t), \quad \forall t\geq 0, 
		$$
		that is
		$${\bf u}(t+\tau+t_{0},\cdot,{\bf u}_0,t_0) \leq  \sigma(\tau,{\bf u}_{0})e^{\lambda_{*}(t+\tau)}{ \varphi}(t+\tau+t_{0},\cdot) \quad t_0\in\mathbb{R}, t\ge 0, \tau\ge 0.$$
		Since the constant $\sigma(\tau,{\bf   u}_0)$ is independent of   $t_0$, then $\sigma(t+\tau,{\bf u}_{0}) \leq \sigma(\tau,{\bf u}_{0})$ for every $t\ge 0$ and $\tau\ge 0$.
	\end{proof}

	Now, we present a proof of Theorem \ref{TH1}.
	
	\begin{proof}[Proof of Theorem \ref{TH1}]
		
		Note from \eqref{def-c} that 
		$$
		{{\bf u}(t+t_0,\cdot;{\bf u}_0,t_0) \leq \sigma(t,{\bf u}_{0})e^{\lambda_{*} t}  { \varphi}(t+t_0,.) \ \text{and}\  {\bf u}(t+t_0,\cdot;{\bf u}_0,t_0) >{\bf 0}\quad \forall\ t>0, \ t_0\in\mathbb{R}.}
		$$
		Thus
		\begin{equation}\label{Q4}  
			\|{\bf u}(t+t_0,\cdot;{\bf u}_0,t_0)\| \leq \sigma(t,{\bf u}_{0})e^{\lambda_{*} t}\| { \varphi}(t+t_0,\cdot)\|_{\infty}
			\le \sigma(t,{\bf u}_{0})e^{\lambda_{*} t} \quad \forall\ t\ge 0.
		\end{equation}
		We distinguish two cases.
		
		\medskip
		
		\noindent{\bf Case 1.} $\lambda_{*}<0$.    The result follows from \eqref{Eq0-5}. It also follows from \eqref{Q4} since $\|{\bf u}(t+t_0,\cdot;{\bf u}_0,t_0)\|\le \sigma(t,{\bf u}_0)e^{t\lambda_*}\le\sigma(0,{\bf u}_0)e^{t\lambda_*}$ for all $t\ge 0$.
		
		\medskip
		
		\noindent{\bf Case 2.} $\lambda_{*}=0$.  We proceed by contradiction to establish that  
		\begin{equation}\label{Q6}
			\lim_{t\to\infty}\sigma(t,{\bf u}_{0}) = 0.
		\end{equation}
		Observe that if \eqref{Q6} holds, then  the desired result follows from \eqref{Q4}. So, it is suffices to establish that \eqref{Q6} holds. Since, $\sigma(\cdot,{\bf u}_0)$ is decreasing in $t\ge 0$ (Lemma \ref{lem1}), then 
		$$
		\sigma(\infty,{\bf u}_{0}):= \lim_{t \rightarrow +\infty}\sigma(t,{\bf u}_{0}) = \underset{t\geq 0}{\inf}\sigma(t,{\bf u}_{0}).
		$$ 
		Suppose to the contrary that  $\sigma(\infty,{\bf u}_{0}) > 0$. First, we note that  
		\begin{equation}\label{Ra4}
			\sup_{t_0\in\mathbb{R}}\max_{i=1,2}\Big\|\frac{u_{i}(t+t_0,\cdot,{\bf u}_0,t_0)}{\varphi_{i}(t+t_0,\cdot)}\Big\|_{\infty} = \sigma(t,{\bf u}_{0})\quad \forall\ t\ge 0.
		\end{equation}
		Hence, by \eqref{Ra4},  for every  $n\ge 1$, there exists $t_{0,n}\in\mathbb{R}$ such that  
		\begin{equation}\label{Ra3}
			\sigma(n,{\bf u}_0)-\frac{1}{n}\le \max_{i=1,2}\Big\|\frac{u_i(n+t_{0,n},\cdot;{\bf u}_0,t_{0,n})}{\varphi_i(n+t_{0,n},\cdot)}\Big\|_{\infty}\le \sigma(n,{\bf u}_0).
		\end{equation}
		Consider the sequences of functions $\{\varphi^n\}_{n\ge 1}$ and $\{{\bf u}^n\}_{n\ge 1}$ defined by 
		\begin{equation*}
			\varphi^n(t,x):={ \varphi}(t_{0,n}+n+t,x)\quad \text{and}\quad {\bf u}^n(t,x):={\bf u}(t_{0,n}+n+t,x;{\bf u}_0,t_{0,n})\quad x\in\Omega,\ t\ge -n,\ n\ge 1.
		\end{equation*}
		Observe that ${\bf u}^n$ satisfies
		\begin{equation}
			\begin{cases}\label{sun}
				\partial_tu_1^{n}=d_1\Delta u_1^{n} +r^nu_2^{n} -s^{n}u_1^{n}-(a^{n}+b^{n}u_1^{n}+c^{n}u_2^{n})u_1^{n} & x\in\Omega,\ t>-n,\cr 
				\partial_tu_2^{n}=d_2\Delta u_2^{n} +s^{n}u^{n}_1 -(e^{n}+f^{n}u_2^{n}+g^{n}u_1^{n})u_2^{n} & x\in\Omega,\ t>-n,\cr 
				0=\partial_{\vec{n}}u_1^{n}=\partial_{\vec{n}}u_2^{n}& x\in\partial\Omega,\ t>-n,
			\end{cases}
		\end{equation}	
		where for each $n\in \mathbb{N}$, $t\in\mathbb{R}$, and $x\in \Omega$,   $$h^n(t,x)=h(t+n+t_{0,n},x),\quad h\in\{r,s,a,b,c,e,f\}.
		$$
		On the other hand, $ \varphi^{n}$ satisfies 		
		\begin{equation}
			\begin{cases}\label{rvarphi}
				\partial_t\varphi_1^{n} =d_1\Delta \varphi_1^{n} +r^{n}\varphi_2^{n} -(s^{n}+a^{n})\varphi_1^{n}& x\in\Omega,\ t\in\mathbb{R},\cr 
				\partial_t\varphi_2^{n} =d_2\Delta \varphi_2^{n} +s^{n}\varphi_1^{n} -e^{n}\varphi_2^{n}& x\in\Omega,\ t\in\mathbb{R},\cr 
				0=\partial_{\vec{n}}\varphi_1^{n}=\partial_{\vec{n}}\varphi_2^{n} & x\in\partial\Omega,\ t\in\mathbb{R}.
			\end{cases}
		\end{equation}
		Thanks to {\bf (H1)} and the Arzela-Ascoli theorem, possibly after passing to a subsequence, we may suppose that there exist  $T$-periodic functions  $h^{\infty}\in C(\mathbb{R}:\overline{\Omega})$, $h\in\{a,b,c,r,s,e,f\}$, satisfying  {\bf (H1)} such that  $h^{n}\to h^{\infty}$ as $n\to \infty$ locally uniformly in $\mathbb{R}\times\bar{\Omega}$. Moreover, since $\sup_{n\ge 1}\|\varphi^n\|\le \| \varphi\|$ and $\sup_{n\ge 1}\|{\bf u}^{n}\|\le M\|{\bf u}_0\|$ (see \eqref{Eq0-5}),  by the regularity properties for parabolic equations \cite[Theorem 3.4.1]{Henry}, after passing to a further subsequence, there exist $\varphi^{\infty}$ and  ${\bf u}^{\infty}$ belonging to  $C(\mathbb{R}\times\bar{\Omega})$ such that $\varphi^{n}\to \varphi^{\infty}$ and  ${\bf u}^{n}\to {\bf u}^{\infty}$ as  $n\to\infty$, locally uniformly. Furthermore, $\varphi^{\infty}$ and ${\bf u}^{\infty}$ are classical solutions of
		\begin{equation}\label{Eq1-1}
			\begin{cases}
				\partial_t\varphi_1^{\infty} =d_1\Delta \varphi_1^{\infty} +r^{\infty}\varphi_2^{\infty} -(s^{\infty}+a^{\infty})\varphi_1^{\infty}& x\in\Omega,\ t\in\mathbb{R},\cr 
				\partial_t\varphi_2^{\infty} =d_2\Delta \varphi_2^{\infty} +s^{\infty}\varphi_1^{\infty} -e^{\infty}\varphi_2^{\infty}& x\in\Omega,\ t\in\mathbb{R},\cr 
				0=\partial_{\vec{n}}\varphi_1^{\infty}=\partial_{\vec{n}}\varphi_2^{\infty} & x\in\partial\Omega,\ t\in\mathbb{R},
			\end{cases}
		\end{equation}	
		and 
		\begin{equation}\label{Eq1-2}
			\begin{cases}
				\partial_t u_1^{\infty} =d_1\Delta u_1^{\infty} +r^{\infty}u_2^{\infty} -(s^{\infty}+a^{\infty}+b^{\infty}u_1^{\infty}+c^{\infty}u_2^{\infty})u_1^{\infty}& x\in\Omega,\ t\in\mathbb{R},\cr 
				\partial_t u_2^{\infty} =d_2\Delta u_2^{\infty} +s^{\infty}1_1^{\infty} -(e^{\infty}+f^{\infty}u_2^{\infty}+g^{\infty}u_1^{\infty})u_2^{\infty}& x\in\Omega,\ t\in\mathbb{R},\cr 
				0=\partial_{\vec{n}}u_1^{\infty}=\partial_{\vec{n}}u_2^{\infty} & x\in\partial\Omega,\ t\in\mathbb{R},
			\end{cases}
		\end{equation}	
		respectively. Recalling that, for each $t>-n$, 
		
		\begin{align*}
			{\bf u}^{n}(t,\cdot)=&{\bf u}(t+n+t_{0,n},\cdot,{\bf u}_0,t_{0,n})
			\le  \sigma(t+n,{\bf u}_0){ \varphi}(t+n+t_{0,n},\cdot)=\sigma(t+n,{\bf u}_0)\varphi^{n}(t,\cdot),
		\end{align*}
		letting $n\to \infty$, we obtain
		\begin{equation}\label{Eq1-3}
			{\bf u}^{\infty}(t,\cdot)\le \sigma(\infty,{\bf u}_0)\varphi^{\infty}(t,\cdot)\quad \forall\ t\in\mathbb{R}.
		\end{equation}
		Noting also from \eqref{Ra3} that
		\begin{equation}\label{Eq1-4}
			\sigma(\infty,{\bf u}_0)=\max_{i=1,2}\Big\|\frac{u^{\infty}(0,\cdot)}{\varphi_i^{\infty}(0,\cdot)}\Big\|_{\infty},
		\end{equation}
		then, since $\sigma(\infty,{\bf u}_0)>0 $, it follows from the comparison for cooperative parabolic systems that   ${\bf u}^{\infty}(t,\cdot)\gg {\bf 0}$ for each  $t\in\mathbb{R}$. We then deduce from  \eqref{Eq1-2} that ${\bf u}^{\infty}$ is a strict sub-solution of \eqref{Eq1-1}. Observing also that $\sigma(\infty,{\bf u}_0)\varphi^{\infty}$ satisfies the cooperative system  \eqref{Eq1-1}, we conclude from the comparison principle for parabolic equations and \eqref{Eq1-3} that 
		$$
		{\bf u}^{\infty}(t,\cdot)\ll \sigma(\infty,{\bf u}_0)\varphi^{\infty}(t,\cdot) \quad \forall\ t\in\mathbb{R},
		$$
		which is contrary to \eqref{Eq1-4}. Hence, \eqref{Q6} holds. 
	\end{proof}

	\subsection{Proofs of Theorem \ref{TH2} and Proposition \ref{Prop1}}
	
	We need a few lemmas.
	
	\begin{lem}\label{Lem2-3}
		For each positive number  $\varepsilon>0$, there exists $\delta_{T,\varepsilon}>0$  such that 
		\begin{equation}
			\sup_{0\le t\le T }\|{\bf u}(t+t_0,\cdot;{\bf u}_0,t_0)\|\le \varepsilon
		\end{equation}
		whenever ${\bf u}(t,\cdot;{\bf u}_0,t_0)$ is a classical solution of \eqref{Eq1} with a positive initial data ${\bf u}_0$ and initial time $t_0\in\mathbb{R}$ satisfying $\|{\bf u}_0\|_{\infty}<\delta_{T,\varepsilon}$. 
		
	\end{lem}
	\begin{proof} 
		Recall from \eqref{Eq0-5} that     $
		\|{\bf u }(t_0+t,\cdot,{\bf u}_0,t_0)\|\le Me^{t\lambda_{*}}\|{\bf u}_0\|$ for all $ t_0\in\mathbb{R}$ and  $ t\ge 0.
		$ So, taking $\displaystyle \delta_{T,\varepsilon}=\frac{\varepsilon e^{-T\lambda_{*}}}{M}$, we get  $
		\sup_{0\le t\le T}\|{\bf u}(t_0+t,\cdot,{\bf u}_0,t_0)\|\le \varepsilon
		$
		whenever  $\|{\bf u}_0\|\le \delta_{T,\varepsilon}$.	
	\end{proof}
	
	The next lemma establishes persistence of solutions with small initial data on the time interval $[0,T]$.
	
	\begin{lem}\label{Lem2-4}  
		There exists $\delta_T>0$ such that given any initial time $t_0\in\mathbb{R}$ and positive initial data ${\bf u}_0$ with $\|{\bf u}_0\|\le \delta _T$, 
		then for every $\gamma>0$ and   $0\le t\le T$, it holds that
		\begin{equation}\label{Main-eq1}
			\gamma{ \varphi}(t+t_0,\cdot;t_0)\le {\bf u}(t+t_0,\cdot;{\bf u}_0,t_0) \quad \text{whenever } \quad  \gamma{ \varphi}(t_0,\cdot)\le {\bf u}_0.
		\end{equation}	
	\end{lem}
	\begin{proof}  Set $K=\max\{\|b\|_{\mathcal{X}_T}+\|c\|_{\mathcal{X}_T},\|f\|_{\mathcal{X}_T}+\|g\|_{\mathcal{X}_T}\}$ and  $\varepsilon=\frac{\lambda_*}{2K}$. So, $\lambda_*-\varepsilon K=\frac{\lambda_*}{2}>0$.  By Lemma \ref{Lem2-3}, there exists $\delta_{T,\varepsilon}>0$ such that for any choice of initial time $t_0\in\mathbb{R}$, if ${\bf u}(t+t_0,\cdot;{\bf u}_0,t_0)$ is a classical solution of \eqref{Eq1} subject to an initial data satisfying $\|{\bf u}_0\|_{\infty}\le \delta_{T,\varepsilon}$, then 
		\begin{equation}\label{Eq2-19}
			\sup_{0\le t\le T}\|{\bf u}(t+t_0,\cdot,{\bf u}_0,t_0)\|\le \varepsilon.
		\end{equation} 
		Let $t_0\in\mathbb{R}$ and  ${\bf u}(t+t_0,\cdot;{\bf u}_0,t_0)$, $t\ge0,$ be a classical solution of \eqref{Eq1} with  $\|{\bf u}_0\|_{\infty}\le \delta_{T}:= \delta_{T,\varepsilon}$.    Observe that
		\begin{equation}\label{Eq2-20}
			\begin{cases} 
				\partial_tu_1=d_{1}\Delta u_1+ru_2 -(a+s+\varepsilon K)u_1+(\varepsilon K-bu_1-cu_2)u_1 & x\in\Omega,\ t>t_0,\cr
				\partial_tu_2=d_2\Delta u_2+su_1-(e+\varepsilon K)u_2+(\varepsilon K-fu_2-gu_1)u_1 & x\in\Omega,\ t>t_0.
			\end{cases}
		\end{equation}
		But  \eqref{Eq2-19} along with the   $T-$periodicity of the parameters of  \eqref{Eq2-20} implies  
		\begin{align*}
			\|b(t,\cdot)u_1(t,\cdot)+c(t,\cdot)u_2(t,\cdot)\|_{\infty}\le&\|b(t,\cdot)\|_{\infty}\|u_1(t,\cdot)\|_{\infty}+\|c(t,\cdot)\|_{\infty}\|u_2(t,\cdot)\|_{\infty}\\
			&\le (\|b\|_{\mathcal{X}_T}+\|c\|_{\mathcal{X}_T})\|{\bf u}(t,\cdot)\|_{\infty}	\le  \varepsilon K
		\end{align*}
		and 
		\begin{align*} 
			\|f(t,\cdot)u_2(t,\cdot)+g(t,\cdot)u_2(t,\cdot)\|_{\infty}\le&\|f(t,\cdot)\|_{\infty}\|u_2(t,\cdot)\|_{\infty}+\|g(t,\cdot)\|_{\infty}\|u_1(t,\cdot)\|_{\infty}\\
			\le& (\|f\|_{\mathcal{X}_T}+\|g\|_{\mathcal{X}_T})\|{\bf u}(t,\cdot)\|_{\infty}\le \varepsilon K
		\end{align*}
		for all $t\in[t_0,T+t_0]$. The last two inequalities together with \eqref{Eq2-20} give
		\begin{equation}
			\begin{cases}\label{Eq2-23}
				\partial_tu_1\ge d_{1}\Delta u_1+ru_2 -(a+s+\varepsilon K)u_1 & x\in\Omega,\ t_0<t<T+t_0,\cr
				\partial_tu_2\ge d_2\Delta u_2+su_1-(e+\varepsilon K)u_2 & x\in\Omega,\ t_0<t<T+t_0,\cr 
				0=\partial_{\vec{n}}{\bf u} & x\in\partial\Omega,\ t_0<t<t_0+T.
			\end{cases}
		\end{equation}
		However, we know that $\Phi(t,x)=e^{\lambda_{d_1,d_2}t}{ \varphi}(t,x)$, where $ \varphi$ is the positive periodic eigenfunction associated with $\lambda_{*}$, solves \eqref{Eq0-2}.	Hence, taking $\lambda_*=\lambda_{d_1,d_2}$, for every  $\gamma>0$,   $$\underline{\bf u}(x,t)=(\gamma e^{(\lambda_*-\varepsilon K)(t-t_0)}\varphi_1(t,\cdot),\gamma e^{(\lambda_*-\varepsilon K)(t-t_0)}\varphi_2(t,\cdot))$$
		satisfies 
		\begin{equation*}
			\begin{cases}
				\partial_t\underline{u}_1= d_{1}\Delta \underline{u}_1+r\underline{u}_2 -(a+s+\varepsilon K)\underline{u}_1 & x\in\Omega,\ t_0<t<T+t_0,\cr
				\partial_t\underline{u}_2= d_2\Delta \underline{u}_2+s\underline{u}_1-(e+\varepsilon K)\underline{u}_2 & x\in\Omega,\ t_0<t<T+t_0,\cr 
				0=\partial_{\vec{n}}\underline{\bf u} & x\in\partial\Omega,\ t_0<t<t_0+T.
			\end{cases}
		\end{equation*}
		As a result, if initially 
		$ 
		{\bf u}_0\ge \gamma{ \varphi}(t_0,\cdot)\,
		$
		we obtain from the comparison principle for parabolic equations that 
		
		\begin{equation}\label{Eq2-23-2}
			{\bf u}(t,\cdot;{\bf u}_0,t_0)\ge \gamma e^{(\lambda_*-\varepsilon K)(t-t_0)}{ \varphi}(t,\cdot)=\gamma e^{\frac{\lambda_*}{2}(t-t_0)}{\varphi}(t,\cdot)\ge\gamma{ \varphi}(t,\cdot)  
			\quad \forall\ t\in[t_0,T+t_0].
		\end{equation}
		In particular
		${\bf u}(t+t_0,\cdot;{\bf u}_0,t_0)\ge \gamma {\varphi}(t+t_0,\cdot)$ whenever $t\in [0,T]$. 
	\end{proof} 
	
	Next, we derive appropriate upper bounds for solutions.
	
	\begin{lem}\label{Lem2-5} There exists $\gamma^*>0$ such that for every $\gamma\ge \gamma^*$, initial time $t_0\in\mathbb{R}$ and initial function ${\bf u}_0\in[{\bf 0},\gamma{\bf \varphi}(t_0,\cdot)]$, it holds that  ${\bf u}(t+t_0,\cdot;{\bf u}_0,t_0)\in[{\bf 0},\gamma{\bf \varphi}(t+t_0,\cdot)]$ for every  $t\ge 0$.
	\end{lem} 
	\begin{proof} Let $t_0\in\mathbb{R}$ and ${\bf u}(t,\cdot;{\bf u}_0,t_0)$ be a classical solution of \eqref{Eq1}. It follows from \eqref{Eq1} that 
		\begin{equation}\label{RE:8}
			\begin{cases}
				\partial_t {\bf u}\le \mathcal{A}(t){\bf u})+\tilde{G}({\bf u}) &  x\in\Omega,\ t>t_0, \cr
				0=\partial_{\vec{n}}{\bf u} &  x\in\partial\Omega,\ t>t_0,
			\end{cases}
		\end{equation}
		where $\tilde{\bf G}$ is  
		$$ 
		\tilde{\bf G}({\bf u})=\left(\begin{array}{c}
			-(b_{\min}u_1)u_1  \\
			-(f_{\min}u_2)u_2  
		\end{array} \right)\quad \forall\ {\bf u}\in [C(\bar{\Omega})]^2.
		$$
		
		Set $\gamma^*=\max\Big\{\frac{\lambda_*}{b_{\min}\varphi_{1,\min}},\frac{\lambda_*}{e_{\min}\varphi_{2,\min}}\Big\}$. Hence, for every $\gamma\ge \gamma^*$,  $t\in [t_0,T+t_0]$ and  $x\in \Omega$
		\begin{align}
			(\mathcal{A}(t)\gamma\varphi)+\tilde{\bf G}(\gamma \varphi)= \partial_{t}(\gamma \varphi)-\lambda_*\gamma\varphi+\tilde{\bf G}(\gamma\varphi)\le \partial_t(\gamma\varphi).
		\end{align}
		As a result, if ${\bf 0}\le {\bf u}_0\le \gamma{ \varphi}(t_0,\cdot)$, it follows from the comparison principle for cooperative systems that 
		$
		{\bf u}(t+t_0,\cdot,{\bf u}_0,t_0)\le \gamma{ \varphi}(t+t_0,\cdot) \quad \forall t\in [0,T].
		$
	\end{proof}
	The last two lemmas establish the existence of invariant rectangles with respect to the uniform-topology. 
	
	\medskip
	
	To state our next result, some notations and definitions are in order. Let $\{e^{t(\Delta -{\rm id})}\}_{t\ge 0}$ denote the analytic $c_0$-semigroup generated by the linear and closed operator $\Delta -{\rm id}$ on $C(\bar{\Omega})$ with domain ${\rm Dom}_{\infty}(\Delta)$. By the maximum principle for parabolic equations, we have that
	\begin{equation}\label{L-infty-estimate}
		\|e^{t(\Delta -{\rm id})}\|\le e^{-t}\quad \forall\ t>0.
	\end{equation}
	Hence, for every $0<\alpha<1$, the fractional power space, denoted  $X^{\alpha}$, of ${\rm id}-\Delta$ is well defined \cite[Theorem 1.4.2]{Henry}. By \cite[Theorem 1.4.3]{Henry}, for every $0<\alpha<1$, there exist $c_{\alpha}>0$ such that 
	\begin{equation}\label{X-alpha-estimate}
		\|e^{t(\Delta -{\rm id})}\|_{X^{\alpha}}\le c_{\alpha}t^{-\alpha}e^{-t}\quad \forall \ t>0.
	\end{equation}
	Since  $\{e^{t(\Delta -{\rm id})}\}_{t>0}$ is compact, then $X^{\beta}$ is compactly embedded in $X^{\alpha}$ for any $0<\alpha<\beta <1$ (\cite[Theorem 1.4.8]{Henry} and  \cite[Theorem 3.3]{Pazy}). Define $\{\bf T(t)\}_{t\ge 0}$ as the analytic $c_0$-semigroup  on $[C(\bar{\Omega})]^2$ given by 
	\begin{equation}\label{T-semigroug-def}
		{\bf T}(t){\bf u}=\Big(e^{td_1(\Delta -{\rm id})}u_1,e^{td_2(\Delta -{\rm id})}u_2\Big)\quad {\bf u}\in [C(\bar{\Omega})]^2,\ t\ge 0.
	\end{equation}
	Thanks to \eqref{X-alpha-estimate}, it holds that 
	\begin{equation}\label{estimate-eq1}
		\|{\bf T}(t)\|_{[X^{\alpha}]^2}\le c_{\alpha}\min\{d_1,d_2\}^{-\alpha}t^{-\alpha}e^{-\min\{d_1,d_2\}t}\quad \forall\ t>0.
	\end{equation}

	\begin{lem}\label{Lem2-6} Fix $0<\alpha<1$ and let $X^{\alpha}$ be the fractional power space defined above. Let $\gamma\ge \gamma^*$ where $\gamma^*>0$ is as in Lemma \ref{Lem2-5}. Then, there exists $K:=K_{T,\alpha,\gamma}>0$ such that for every  $t_0 \in \mathbb{R}$, if  ${\bf u}_0\in [{\bf 0},\gamma{ \varphi}(t_0,\cdot)]\cap [X^{\alpha}]^2$ satisfies $\|{\bf u}_0\|_{[X^{\alpha}]^2}\le K$, then  $\|{\bf u}(T+t_0,\cdot,{\bf u}_0,t_0)\|_{[X^{\alpha}]^2}\le K$. 
		
	\end{lem}
	\begin{proof} Let ${\bf u}(t,\cdot;{\bf u}_0,t_0)$ be a classical solution of  \eqref{Eq1}. By the variation of constant formula, 
		\begin{equation*}
			{\bf u}(t+t_0,\cdot;{\bf u}_0,t_0)={\bf T}(t){\bf u}_0+\int_{t_0}^{t+t_0}{\bf T}(t
			+t_0-s){\bf \tilde{F}}(s,\cdot,{\bf u}(s,\cdot;{\bf u}_0,t_0))ds \quad \forall\ t\ge 0.
		\end{equation*} 
		where $\{{\bf T}(t)\}_{t\ge 0}$ is the analytic $c_0$-semigroup defined in \eqref{T-semigroug-def} and 
		$${\bf \tilde{F}}(t,\cdot,{\bf u})=\left( \begin{array}{c}
			r(t,.)u_2-s(t,.)u_1+d_1u_1-(a(t,.)+b(t,.)u_1+c(t,.)u_2)u_1\\
			s(t,.)u_1-(e(t,.)+f(t,.)u_2+g(t,.)u_1)u_2+d_2u_2
		\end{array}
		\right)\quad {\bf u}\in [C(\bar{\Omega})]^2,\ t\in\mathbb{R}.  $$  
		A change of variable yields
		$$
		{\bf u}(t+t_0,\cdot;{\bf u}_0,t_0)= {\bf T}(t){\bf u}_0+\int_{0}^{t}{\bf T}(t
		-s){\bf \tilde{F}}(s+t_0,\cdot, {\bf u}(s+t_0,\cdot;{\bf u}_0,t_0))ds, \quad t>0.
		$$ 
		Hence, by \eqref{estimate-eq1}, setting  $\delta:=\min\{d_1,d_2\}$,  for every $t>0$,
		\begin{align}\label{RE:9}
			&\|{\bf u}(t+t_0,\cdot;{\bf u}_0,t_0)\|_{[X^{\alpha}]^2}\cr 
			\le& \frac{c_{\alpha}}{\delta^\alpha}\Big(t^{-\alpha}e^{-\delta t}\|{\bf u}_{0}\|_{\infty}+\int_{0}^{t}(t-s)^{-\alpha}e^{-\delta (t-s)}\|{\bf \tilde{F}}(s+t_0,\cdot,{\bf u}(s+t_0,\cdot;{\bf u}_0,t_0))\|_{\infty}ds\Big).
		\end{align}
		Now suppose that ${\bf u}_0\in[{\bf 0},\gamma{ \varphi}(t_0,\cdot)]$. Then  ${\bf u}(t+t_0,\cdot;{\bf u}_0,t_0)\in[{\bf 0},\gamma{ \varphi}(t+t_0,\cdot)]$ for every  $t>0$ by Lemma \ref{Lem2-5}. This implies that 
		$$\|{\bf u}(t+t_0,\cdot;{\bf u}_0,t_0)\|\le \gamma\|{ \varphi}(t+t_0,\cdot)\|\le \gamma\|\varphi\|_{[\mathcal{X}_T]^2} =\gamma\quad \forall t\ge 0.$$
		Therefore, for every $t\ge t_0$,
		\begin{align*}
			&\|(b(t,\cdot)u_1(t,\cdot;{\bf u}_0,t_0)+c(t,\cdot)u_2(t,\cdot;{\bf u}_0,t_0))u_1(t,\cdot;{\bf u}_0,t_0)\|_{\infty}\cr
			\le& (\|b(t,\cdot)\|_{\infty}\gamma+\|c(t,\cdot)\|_{\infty}\gamma)\gamma
			\le (\|b\|_{\mathcal{X}_T}+\|c\|_{\mathcal{X}_T})\gamma^{2}
		\end{align*}
		and 
		$$ 
		\|(e(t,\cdot)u_2(t,\cdot;{\bf u_0},t_0)+g(t,\cdot)u_1(t,\cdot;{\bf u}_0,t_0))u_2(t,\cdot;{\bf u}_0,t_0)\|_{\infty}\le (\|e\|_{\mathcal{X}_T}+\|g\|_{\mathcal{X}_T})\gamma^{2} .
		$$
		So, introducing the bounded linear operator  ${\bf B}(t) : [C(\bar{\Omega})]^2\to [C(\bar{\Omega})]^2$ \begin{equation*}
			{\bf B}(t,\cdot, {\bf u})=\left( 
			\begin{array}{c}
				r(t,\cdot)u_2+(d_1-s(t,\cdot) -a(t,\cdot) )u_1  \\
				s(t,\cdot)u_1+(d_2-e(t,\cdot))u_2
			\end{array}
			\right)\quad \forall\ {\bf u}\in X^2,\ t\in\mathbb{R},
		\end{equation*}
		we obtain 
		$$ 
		\|{\bf B}\|:=\sup_{t\in\mathbb{R}}\|{\bf B}(t)\|\le \|r\|_{\mathcal{X}_T}+\|s\|_{\mathcal{X}_T}+\|a\|_{\mathcal{X}_T}+\|e\|_{\mathcal{X}_T}+d_1+d_2\quad \forall\ t\in\mathbb{R},
		$$
		and 
		$$ 
		\|{\tilde{\bf F}}(t,\cdot, {\bf u}(t,\cdot;{\bf u}_0,t_0))\|_{\infty}\le (\|{\bf B}\|+\|b\|_{\mathcal{X}_T}+\|c\|_{\mathcal{X}_T}+\|e\|_{\mathcal{X}_T}+\|g\|_{\mathcal{X}_T})\gamma^{2}:=M_0\gamma^{2}\quad \forall\ t\ge t_0.
		$$
		Combining this with \eqref{RE:9}, we get
		
		\begin{align}\label{RE:10}
			\|{\bf u}(t+t_0,\cdot;{\bf u}_0,t_0)\|_{[X^{\alpha}]^2}\le & \frac{c_{\alpha}}{\delta^{\alpha}}\gamma\Big(t^{-\alpha}e^{-\delta t}+M_0\gamma\int_0^{t}(t-s)^{-\alpha}e^{-\delta (t-s)}ds\Big)\cr
			\le &\frac{c_{\alpha}}{\delta^{\alpha}}\gamma\Big(t^{-\alpha}e^{-\delta t}+M_0\gamma\int_0^{\infty}s^{-\alpha}e^{-\delta s}ds\Big)\quad \forall\ t>0.
		\end{align}
		As a result, we may take 
		$K_{T,\alpha,\gamma}=\frac{c_{\alpha}}{\delta^{\alpha}}\gamma\Big(T^{-\alpha}e^{-\delta T}+M_0\gamma\int_0^{\infty}s^{-\alpha}e^{-\delta s}ds\Big)$.
	\end{proof}
	
	Combining all the previous results, we get the next result.
	
	\begin{lem}\label{Lem2-7} Fix $0<\alpha<1$ and let $X^{\alpha}$ be the fractional power space in Lemma \ref{Lem2-6}. Let $\gamma\ge \gamma^*$ where  $\gamma^*$ is given by Lemma \ref{Lem2-5}. Let $K=K_{T,\alpha,\gamma}$ be the positive number given by Lemma \ref{Lem2-6}.
		For  $0<\xi<\gamma$ 
		and $t_0\in\mathbb{R}$, define 
		\begin{equation}\label{RE:18}
			\mathcal{M}^{t_0,\alpha}_{\gamma,\xi}:=\{{\bf u}_0\in [\xi{ \varphi}(t_0,\cdot),\gamma{ \varphi}(t_0,\cdot)]\cap [X^{\alpha}]^2 \ : \ \|{\bf u}_0\|_{[X^{\alpha}]^2}\le K \}.
		\end{equation}
		
		Then, there exists $\xi^*=\xi^*(T,\alpha,\gamma)>0$ such that for every $\xi\in(0,\xi^*]$, initial  time $t_0\in\mathbb{R}$, and initial distribution ${\bf u}_0\in \mathcal{M}^{t_0,\alpha}_{\gamma,\xi}$,   the   solution ${\bf u}(\cdot,\cdot;{\bf u}_0,t_0)$ of  \eqref{Eq1} 
		satisfies ${\bf u}(t_0+T,\cdot;{\bf u}_0,t_0)\in\mathcal{M}^{t_0,\alpha}_{{\gamma},\xi}$.
		
	\end{lem}
	
	\begin{proof} 
		
		We proceed by contradiction to establish the existence of  $\xi^*$. To this end, suppose that there exists a sequence of positive numbers $(\xi_n)_n$ converging to zero and a sequence of initial times $\{t_0^n\}_{n\ge 1}$, and initial distributions $\{{\bf u}_{0,n}\}_{n\ge 1}$ such that  ${\bf u}_{0,n}\in\mathcal{M}^{t_0^n,\alpha}_{\gamma,\xi_n}$  and  $
		{\bf u}(t_0^n+T,\cdot;{\bf u}_{0,n},t_0^n)\notin \mathcal{M}^{t_0^n,\alpha}_{{\gamma},\xi_n}$ for every $ n\ge 1.$    Since ${\bf u}_{0,n}\in[{\bf 0},\gamma{ \varphi}(t_0^n,\cdot)]\cap [X^{\alpha}]^2$, $\|{\bf u}_{0,n}\|_{[X^{\alpha}]^2}\le K$, and $\varphi$ is $T$-periodic, then by Lemma \ref{Lem2-6}, for every  $n\ge 1$, ${\bf u}(t_0^n+T,\cdot;{\bf u}_{0,n},t_0^n)\in[{\bf 0},\gamma{ \varphi}(t_0^n,\cdot)]\cap [X^{\alpha}]^2$ and  $\|{\bf u}(t_0^n+T,\cdot;{\bf u}_{0,n},t_0^n)\|_{[X^{\alpha}]^2}\le K$ . As a result, for each $n\ge 1$, there exists $x_{n}\in\Omega$ such that either  $u_{1}(t_0^n+T,x_n;{\bf u}_{0,n},t_0^n)<\xi_n{ \varphi}(t_0^n,x_1)$ or $u_2(t_0^n+T,x_n;{\bf u}_{0,n},t_0^n)<\xi_n{ \varphi}(t_0^n,x_n)$. This implies that
		\begin{equation}\label{RE:13}
			\min\{u_{1,\min}(t_0^n+T,\cdot;{\bf u}_{0,n},t_0^n),u_{2,\min}(t_0^n+T,\cdot;{\bf u}_{0,n},t_0^n)\}<{\xi_n}\|\varphi\|_{[\mathcal{X}_T]^2}={\xi_n}\quad \forall\ n\ge 1.
		\end{equation}
		Moreover, recalling Lemma \ref{Lem2-4}, we also have 
		\begin{equation}\label{RE:12}
			\|{\bf u}_{0,n}\|_{\infty}\ge \delta_T\quad \forall\ n\ge 1,
		\end{equation}
		where $\delta_T$ is the positive number given by Lemma \ref{Lem2-4}.  Since there is $0<\nu\ll 1$ such that $X^{\alpha}$ is compactly embedded in $C^{\nu}(\overline{\Omega})$ (\cite[Theorem 1.6.1]{Henry}), and given that  $\|{\bf u}_0^n\|_{X^{\alpha}\times X^{\alpha}}\le K$ for every $n\ge 1$, we can invoke the  Arzela-Ascoli theorem to extract a subsequence $\{{\bf u}_{0,n'}\}_{n'\ge 1}$ and  ${\bf u}_0\in C^{\nu}(\overline{\Omega})$ such that 
		\begin{equation}\label{RE13}
			\lim_{n\to\infty}\|{\bf u}_{0,n'}-{\bf u}_0\|_{C^{\nu}(\bar{\Omega})}=0,\quad 0<\nu\ll 1.
		\end{equation}
		Furthermore, thanks to {\bf (H1)}, possibly after passing to another subsequence, for each $\tau\in\{r,s,a,b,c,e,f, g\}$, we have $\tau(t_{0}^n+t,\cdot)\to \tau^{\infty}(t,\cdot)$ locally uniformly in $\mathbb{R}\times C(\bar{\Omega})$ as $n\to\infty$, where $\tau^{\infty}$ is $T$-periodic. In addition,  by the similar arguments leading to \eqref{Eq1-2} and recalling \eqref{RE:13}, we may suppose that   ${\bf u}(t+t_{0}^n,\cdot;{\bf u}_{0,n},t_{0}^n)\to {\bf u}^{\infty}(t,\cdot)$ locally uniformly in $\mathbb{R}^+\times \bar{\Omega}$ as $n\to \infty$, and ${\bf u}^{\infty}(t,x)$ satisfies 
		\begin{equation}\label{DD6}
			\begin{cases}
				\partial_tu_1^{\infty}=d_1\Delta u_1^{\infty}+r^{\infty}(t,\cdot)u_2^{\infty}-(a^{\infty}(t,\cdot)+b^{\infty}(t,\cdot)u_1^{\infty}+c^{\infty}u_2^{\infty})u_1^{\infty} & x\in\Omega,\ t>0,\cr 
				\partial_tu_2^{\infty}=d_2\Delta u_2^{\infty}+s^{\infty}u_1^{\infty}-(e^{\infty}(t,\cdot)+f^{\infty}(t,\cdot)u_2^{\infty}+g^{\infty}(t,\cdot)u_1^{\infty}(t,\cdot))u_2^{\infty} & x\in\Omega,\ t>0,\cr 
				0=\partial_{\vec{n}}{\bf u}^{\infty} & x\in\partial\Omega,\ t>0,\ \cr
				{\bf u}^{\infty}(0,\cdot)={\bf u}_0 & x\in\bar{\Omega}.
				
			\end{cases}
		\end{equation}
		Since ${\bf u}_{0,n}\in [C^{+}(\bar{\Omega})]^2$ for each $n\ge 1,$ then ${\bf u}_{0}\in [C^+(\bar{\Omega})]^2$. Note also from \eqref{RE:12} that $\|{\bf u}_{0}\|_{\infty}\ge \delta_T$. Therefore, thanks to the comparison principle for cooperative systems, ${\bf u}^{\infty}(t,\cdot)\in [C^{++}(\bar{\Omega})]^2$ for every $t>0.$ In particular 
		\begin{equation}\label{RE:15}
			\min_{x\in\bar{\Omega}}\min\{u_1^{\infty}(T,x),u^{\infty}_2(T,x)\}>0.
		\end{equation}
		However, since ${\bf u}(T+t_{0,n},\cdot;{\bf u}_{0,n},t_0^n)\to {\bf u}^{\infty}(T,\cdot)$ as $n\to\infty$ in $C(\bar{\Omega})$, we deduce from \eqref{RE:13} that $\min_{x\in\bar{\Omega}}\min\{u_1^{\infty}(T,x),u_2^{\infty}(T,x)\}=0$, which is contrary to \eqref{RE:15}. Hence the desired result.
	\end{proof}
	
	Now, we give a proof of Theorem \ref{TH2}.
	
	\begin{proof}[Proof of Theorem \ref{TH2}]
		Let ${\bf u}(t,\cdot;{\bf u_0},t_0)$ be a classical solution of \eqref{Eq1} subject to a positive initial data ${\bf u}_0$ and initial time $t_0\in\mathbb{R}$. We first show that there exists $\gamma_0>0$ such that 
		
		\begin{equation}\label{DD2}
			\gamma_0\le \inf_{t_0\in \mathbb{R}}\min\Big\{\min_{x\in\bar{\Omega}}\frac{u_1(T+t_0,x;{\bf u}_0,t_0)}{\varphi_1(t_0+T,x)},\min_{x\in\bar{\Omega}}\frac{u_1(T+t_0,x;{\bf u}_0,t_0)}{\varphi_1(t_0+T,x)}\Big\}.
		\end{equation}
		Indeed, since 
		$\lim_{t\to 0^+}\sup_{t_0\in\mathbb{R}}\|{\bf u}(t_0+t,\cdot; {\bf u}_0,t_0)-{\bf u}_0\|=0$ 
		and $\|{\bf u}_0\|>0$, there exists $0<t^*\ll T$ such that 
		\begin{equation}\label{DD3}
			\inf_{t_0\in\mathbb{R}}\|{\bf u}(t_0+t^*,\cdot;{\bf u}_0,t_0)\|>0.
		\end{equation}
		Now, we proceed by contradiction to establish \eqref{DD2}. To this end, we suppose that there is a sequence of initial times $\{t_{0,n}\}_{n\ge 1}$ satisfying 
		\begin{equation}\label{DD4}
			\lim_{n\to\infty}\min\Big\{\min_{x\in\bar{\Omega}}\frac{u_1(T+t_{0,n},x;{\bf u}_0,t_{0,n})}{\varphi_1(t_{0,n}+T,x)},\min_{x\in\bar{\Omega}}\frac{u_1(T+t_{0,n},x;{\bf u}_0,t_{0,n})}{\varphi_1(t_{0,n}+T,x)}\Big\}=0.
		\end{equation}
		Consider the sequence of functions $\{{\bf u}^n(\cdot,\cdot) \}_n\ge 1$ defined by ${\bf u}^n(t,x)={\bf u}(t_0+t^*+t,x,{\bf u}_0,t_{0,n})$, $t\ge 0$ and $x\in\bar{\Omega}$. Since, $\sup_{n\ge 1}\|{\bf u}^n\|\le \sup_{t\ge t_0}\|{\bf u}(t,\cdot;{\bf u}_0,t_0)\|<\infty $ and $\sup_{n\ge 0}\|{\bf u}^n(0,\cdot)\|_{[X^{\alpha}]^2}<\infty$ (see \eqref{RE:10} for the last inequality), then, possibly after passing to a subsequence, we can employ the regularity theory for parabolic equations to conclude that there exist ${\bf u}_0^{\infty}\in C^{\nu}(\bar{\Omega})$, $0<\nu\ll1$ and  ${\bf u}^{\infty}$ satisfying \eqref{DD6}, such that ${\bf u}^n(t,\cdot)\to {\bf u}^{\infty }(t,\cdot)$ as $n\to\infty$, locally uniformly on $\mathbb{R}^+\times\bar{\Omega}$, and ${\bf u}^{\infty}(0,\cdot)={\bf u}^{\infty}_0$. Recalling that $\min_{i=1,2}\varphi_{i,\min}>0$, we deduce from \eqref{DD4} that
		$$
		\min_{x\in\bar{\Omega}}\{u_1^{\infty}(T-t^*,x),u_2^{\infty}(T-t^*,x)\}=0,
		$$
		which in view of the maximum principle for cooperative parabolic systems implies that $\|{\bf u}^{\infty}(T-t^*,\cdot)\|=0$. Thus, ${\bf u}^{\infty}(t,\cdot)={\bf 0}$ for all $t\ge 0$. In particular, ${\bf u}^{\infty}_0={ 0}$. However, by \eqref{DD3}, we have that $\|{\bf u}^{\infty}_0\|>0$. So, we have a contradiction. Therefore, \eqref{DD2} holds.
		
		\medskip
		
		Let $\delta_T$ be the positive number of Lemma \ref{Lem2-3} and set $   
		\gamma=\min\Big\{\delta_T,\gamma_0\Big\}$, so that 
		${\bf u}(t_0+T,\cdot;{\bf u}_0,t_0)\ge \gamma { \varphi}(t_0+T,\cdot)=\gamma{\varphi}(t_0,\cdot)$  for any  $t_0\in \mathbb{R}$. Then,  by Lemma  \ref{Lem2-4}, 
		\begin{equation}\label{ERT1}
			{\bf u}(t_0+nT,\cdot;{\bf u}_0,t_0)\ge \gamma { \varphi}(t_0+nT,\cdot)=\gamma { \varphi}(t_0,\cdot)\quad \forall\ n\ge 1 \quad \text{and} \quad t_0\in\mathbb{R}.
		\end{equation}
		Finally, taking 
		$$ 
		K:=1+\sup_{t\ge 0, t_0\in\mathbb{R}}\|{\bf u}(t+t_0,\cdot,{\bf u}_0,t_0)\|_{\infty}+\|a\|_{\mathcal{X}_T}+\|b\|_{\mathcal{X}_T}+\|c\|_{\mathcal{X}_T}+\|e\|_{\mathcal{X}_T}+\|f\|_{\mathcal{X}_T}<\infty,
		$$ 
		we obtain 
		\begin{equation*}
			\partial_tu_i\ge d_i\Delta u_i-K^2u_i \quad \forall\ t\ge 0,\ i=1,2. 
		\end{equation*}
		We can now employ the comparison principle for parabolic equations to conclude that, for each   $n\ge 1$ and  $0\le t\le 1$,
		$$ 
		{ u}_{i}(t_0+nT+t,x;{\bf u}_0,t_0)\ge e^{-tK^2}\min_{x\in\overline{\Omega}}{ u}_{i}(t_0+nT,x;{\bf u}_0,t_0)\quad \forall\ i=1,2. 
		$$
		Taking $\eta_*:=\gamma e^{-K^2}\min_{i=1,2}\varphi_{i,\min}$, the last inequality along with  \eqref{ERT1} implies that  
		${\bf u}(t_0+t,\cdot;{\bf u_0},t_0)\ge (\eta_*,\eta_*)$  for all $t\ge T$ and $t_0\in\mathbb{R}$. 
	\end{proof}
	
	We end this subsection with a proof of Proposition \ref{Prop1}. 
	
	\begin{proof}[Proof of Proposition \ref{Prop1}] Suppose that {\bf (H1)} holds. Let $\lambda_*$ be the principal eigenvalue of \eqref{Eq2} and $\varphi$ be a corresponding positive $T$-periodic eigenfunction. Now, we proceed to prove {\rm (i)} and {\rm (ii)}.
		
		\medskip
		
		{\rm (i)} First, suppose that $(r+s)^2\le,\ne 4(a+s)e$. Set $\Omega_T:=[0,T]\times\Omega$ and $\mathcal{O}:=\{(t,x)\in \Omega_T : (a+s)e>0\}$.  Hence $\mathcal{O}^c:=\Omega_T\setminus\mathcal{O}\subset\{(t,x)\in \Omega_T : (r+s)=0\}$. 
		Next, observe that 
		$$
		\begin{cases}
			\lambda_*\varphi_1^2+\frac{1}{2}\frac{d}{dt}\varphi_1^2=d_1\varphi_1\Delta\varphi_1+r\varphi_1\varphi_2-(a+s)\varphi_1^2 & x\in\Omega,\ 0<t<T,\cr
			\lambda_*\varphi_2^2+\frac{1}{2}\frac{d}{dt}\varphi_2^2=d_2\varphi_2\Delta \varphi_2+s\varphi_1\varphi_2-e\varphi_2^2 & x\in\Omega,\ 0<t<T,\cr 
			0=\partial_{\vec{n}}\varphi_1=\partial_{\vec{n}}\varphi_2 & x\in\partial\Omega,\ 0<t<T.
		\end{cases}
		$$
		Integrating the first two equations, we obtain
		$$ 
		\lambda_*\int_{\Omega_T}\varphi_1^2+\frac{1}{2}\int_{\Omega}(\varphi_1^2(T,\cdot)-\varphi_1^2(0,\cdot)\Big)=-d_1\int_{\Omega_T}|\nabla\varphi_1|^2+\int_{\Omega_T}r\varphi_1\varphi_2-\int_{\Omega_T}(a+s)\varphi_1^2
		$$
		and 
		$$ 
		\lambda_*\int_{\Omega_T}\varphi_2^2+\frac{1}{2}\int_{\Omega}(\varphi_2^2(T,\cdot)-\varphi_2^2(0,\cdot))=-d_2\int_{\Omega_T}|\nabla \varphi_2|^2+\int_{\Omega_T}s\varphi_1\varphi_2-\int_{\Omega_T}e\varphi_2^2.
		$$
		It follows from the last two equations and the fact that $r+s=0$ on $\mathcal{O}^c$, that 
		\begin{align*}
			\lambda_*\int_{\Omega_T}(\varphi_1^2+\varphi_2^2)\le&- \int_{\Omega_T}\left((a+s)\varphi_1^2-(r+s)\varphi_1\varphi_2+e\varphi_2^2\right)\cr 
			=&- \int_{\mathcal{O}^c}\left((a+s)\varphi_1^2-(r+s)\varphi_1\varphi_2+e\varphi_2^2\right)- \int_{\mathcal{O}}\left((a+s)\varphi_1^2-(r+s)\varphi_1\varphi_2+e\varphi_2^2\right)\cr
			\le &- \int_{\mathcal{O}}\left((a+s)\varphi_1^2-(r+s)\varphi_1\varphi_2+e\varphi_2^2\right)\cr
			=&-\int_{\mathcal{O}}\Big(\sqrt{a+s}\varphi_1-\frac{r+s}{2\sqrt{a+s}}\varphi_2\Big)^2+\int_{\mathcal{O}}\Big(\frac{(r+s)^2}{4(a+s)}-e\Big)\varphi_2^2.
		\end{align*}
		Therefore, since $e-\frac{(r+s)^2}{4(a+s)} \ge,\ne 0$ on $\mathcal{O}$ and $\varphi>>{\bf 0}$, we deduce from the last inequality that  $\lambda_*<0$.
		
		\medskip 
		
		Next, suppose that $1<\big(\frac{a+s}{r}\big)_{\min}\big(\frac{e}{s}\big)_{\min}$. Then $e_{\min}>0$. Take $l_0:=\frac{1}{2}\Big(\frac{1}{\big(\frac{e}{s}\big)_{\min}}+\big(\frac{s+a}{r}\big)_{\min}\Big)$. Hence, there is $\varepsilon_0>0$ such that  $\frac{1}{\big(\frac{e}{s}\big)_{\min}}<l_0-\varepsilon_0<l_0<l_0+\varepsilon_0<\big(\frac{s+a}{r}\big)_{\min}$. Next, set $\tilde\varphi_1\equiv 1$ and $\tilde\varphi_2=l_0\tilde\varphi_1$. Then, since $l_0+\varepsilon_0<\big(\frac{a+s}{r}\big)_{\min}$, we have 
		$$
		r\tilde\varphi_2-(s+r)\tilde\varphi_1<-\varepsilon_0r \le -\varepsilon_0r_{\min}\tilde\varphi_1 \quad x\in\bar{\Omega},\ 0\le t\le T.
		$$
		Similarly, since $1 <(l_0-\varepsilon_0)\big(\frac{e}{s}\big)_{\min}$ and $\tilde\varphi_2=l_0\tilde\varphi_1$, we have
		$$
		s\tilde\varphi_1-e\tilde\varphi_2<-e\varepsilon_0\tilde\varphi_1=\frac{e\varepsilon_0}{l_0}\tilde\varphi_2\le-\frac{e_{\min}\varepsilon_0}{l_0}\tilde\varphi_2\quad x\in\bar{\Omega},\ 0\le t\le T.
		$$
		Therefore, since $\tilde\varphi=(\tilde\varphi_1,\tilde\varphi_2)$ is constant and positive, hence $T$-periodic and satisfies the homogeneous Neumann boundary conditions, we can employ the comparison principe for principle eigenvalue of linear cooperative system to deduce that $\lambda_*\le -\min\{\varepsilon_0r_{\min},\frac{e_{\min}\varepsilon_0}{l_0}\}<0$ for any choice of diffusion rates $d_1>0$ and $d_2>0$.

		\medskip
		
		{\rm(ii)} Observe that 
		$$
		\begin{cases}
			\lambda_*+\partial_t\ln(\varphi_1)=\frac{d_1}{\varphi_1}\Delta\varphi_1+r\frac{\varphi_2}{\varphi_1}-(a+s) & x\in\Omega,\ 0<t<T,\cr
			\lambda_*+\partial_t\ln(\varphi_2)=\frac{d_2}{\varphi_2}\Delta\varphi_2+s\frac{\varphi_1}{\varphi_2}-e & x\in\Omega,\ 0<t<T,\cr
			0=\partial_{\vec{n}}\varphi_1=\partial_{\vec{n}}\varphi_2 & x\in\partial\Omega, 0<t<T.
		\end{cases}
		$$
		Hence, since $\varphi$ is $T$-periodic, integrating the first two equations and add up the resulting equations, we obtain
		\begin{align*}
			2\lambda_*T|\Omega|=&\int_{\Omega}\sum_{i=1}^2\ln\left(\frac{\varphi_i(0,\cdot)}{\varphi_i(T,\cdot)}\right)+\int_{\Omega_T}\sum_{i=1}^2d_i|\nabla \ln\varphi_i|^2 +\int_{\Omega_T}\Big(r\frac{\varphi_2}{\varphi_1}+s\frac{\varphi_1}{\varphi_2}\Big)-\int_{\Omega_T}(a+s+e)\cr
			=&\int_{\Omega_T}\sum_{i=1}^2d_i|\nabla \ln\varphi_i|^2 +\int_{\Omega_T}\Big(r\frac{\varphi_2}{\varphi_1}+s\frac{\varphi_1}{\varphi_2}\Big)-\int_{\Omega_T}(a+s+e)\cr
			\ge & \int_{\Omega_T}\Big(r\frac{\varphi_2}{\varphi_1}+s\frac{\varphi_1}{\varphi_2}\Big)-\int_{\Omega_T}(a+s+e)\cr
			\ge& \int_{\Omega_T}(2\sqrt{rs}-(a+s+e)).
		\end{align*}
		Hence $\lambda_*\ge\frac{1}{T|\Omega|}\int_{\Omega_T}(2\sqrt{rs}-(a+s+e))$. It then follows that $\lambda_*>0$ if $ \int_{\Omega_T}\sqrt{rs}>\frac{1}{2}\int_{\Omega_T}(a+s+e))$.  
		
		\medskip
		
		Next,  suppose that $1>\big(\frac{a+s}{r}\big)_{\max}\big(\frac{e}{s}\big)_{\max}$. We distinguish two cases.\\
		{\bf Case 1.}  $e\equiv0$. Then integrating the second equation of \eqref{Eq2}, we obtain
		$$
		\lambda_*\int_{\Omega_T}\varphi_2=\int_{\Omega}s\varphi_1>0,
		$$
		which implies that $\lambda_*>0$. \\
		{\bf Case 2.} $e\ge,\ne 0$. Choose $l_1>0$ and $\varepsilon_1>0$ satisfying $\frac{1}{\big(\frac{a+s}{r}\big)_{\max}}>l_1+\varepsilon_1>l_1>l-\varepsilon_1>\big(\frac{e}{s}\big)_{\max}$. Set $\tilde{\varphi}_2=1$ and $\tilde{\varphi}_1=l_1$. Then, since $1>(l_1+\varepsilon_1)\big(\frac{a+s}{r}\big)_{\max}$, we have
		$$
		r\tilde{\varphi}_2-(a+s)\tilde{\varphi}_1>\varepsilon_1(a+s)>\frac{\varepsilon_1(a+s)_{\min}}{l_1}\tilde{\varphi}_1\quad x\in\bar{\Omega},\ 0\le t\le T.
		$$
		Similarly, since $(l_1-\varepsilon_1)>\big(\frac{e}{s}\big)_{\max}$, we have 
		$$
		s\tilde{\varphi}_1-e\tilde{\varphi}_2>\varepsilon_1s\ge \varepsilon_1s_{\min}\tilde{\varphi}_2 \quad x\in\bar{\Omega},\ 0\le t\le T.
		$$
		Therefore, since $\tilde{\varphi}=(\tilde{\varphi}_1,\tilde{\varphi}_2)$ is constant and positive, then by the comparison principle for principal eigenvalues or linear cooperative systems, we have that $\lambda_*\ge \min\{\varepsilon_1s_{\min},\frac{\varepsilon_1(a+s)_{\min}}{l_1}\}>0$ for any choice of diffusion rates $d_1>0$ and $d_2>0$.
	\end{proof}
	
	\subsection{Proof of Theorem \ref{TH3}}
	
	\begin{proof}[Proof of Theorem \ref{TH3}]
		When the parameters are time-independent, the existence of a positive steady state of \eqref{Eq1} when $\lambda_*>0$ is proved in \cite{CCM2020}. The existence of a constant positive positive equilibrium for the kinetic system when $\lambda_*>0$ is proved in \cite{bouguima2012asymptotic}. Hence, we shall only prove the existence of a time-periodic solution when the environment is time-periodic and spatially homogeneous.
		
		\medskip

		\noindent Fix $0<\alpha<1$ and $\gamma^*$ the positive number obtained in Lemma \ref{Lem2-5}. Let $K=K_{T,\alpha,\gamma^*}$ be given by Lemma \ref{Lem2-6}, and $\xi^*>0$ be provided by Lemma \ref{Lem2-7}. Consider the closed bounded and convex subset of $[X^{\alpha}]^2$,  given by $\mathcal{M}^{0,\alpha}_{\gamma^*,\xi^*}$, where the set $\mathcal{M}^{T-0,\alpha}_{\gamma^*,\xi^*}$ is defined by \eqref{RE:18} for every $t_0\in\mathbb{R}$ (Here we take $t_0=0$). Hence, by Lemma \ref{Lem2-7}, we have that ${\bf u}(T,\cdot;{\bf u}_0,0)\in\mathcal{M}^{0,\alpha}_{\gamma^*,\xi^*}$ whenever ${\bf u}_0\in\mathcal{M}^{0,\alpha}_{\gamma^*,\xi^*}$. This shows that the Poincar\'e map ${\bf u}(T,\cdot): \mathcal{M}^{0,\alpha}_{\gamma^*,\xi^*}\ni {\bf u}_0\mapsto {\bf u}(T,\cdot,{\bf u}_0,0)\in\mathcal{M}^{0,\alpha}_{\gamma^*,\xi^*}$ is well defined. 
		Note from \eqref{RE:10} that, for all $\beta\in(\alpha,1)$,
		$$
		\|{\bf u}(T,\cdot,{\bf u}_0,0)\|_{[X^{\beta}]^2}\le \frac{c_{\beta}}{\delta^{\beta}}\gamma^*\Big(T^{-\beta}e^{-\delta T}+M_0\gamma^*\int_{0}^{\infty}s^{-\beta}e^{-\delta s}ds\Big)\quad \forall\ {\bf u}_0\in\mathcal{S}^{*}_T.
		$$
		Since $X^{\beta}$ is compactly embedded in $X^{\alpha}$ for all $\beta\in(\alpha,1)$, the Poincar\'e map ${\bf u}(T,\cdot)$ is compact  on $\mathcal{M}^{0,\alpha}_{\gamma^*,\xi^*}$. Moreover, given that $\tilde{F}$ is locally lipschitz in ${\bf u}$, it follows from the regularity theory for parabolic equations that  the Poincar\'e map  ${\bf u}(T,\cdot,{\bf u_0},0)$ is continue with respect to ${\bf u_0}$ in $\mathcal{M}^{0,\alpha}_{\gamma^*,\xi^*}$. Thus, thanks to the Schauder fixed point theorem, the function ${\bf u}(T,\cdot,\cdot,0)$ admits a fixed point in $\mathcal{M}^{0,\alpha}_{\gamma^*,\xi^*}$, say ${\bf u}^{T,*}$, that is   there exists ${\bf u}^{T,*}\in\mathcal{M}^{0,\alpha}_{\gamma^*,\xi^*}$  satisfying ${\bf u}(T,\cdot,{\bf u}^{T,*},0)={\bf u}^{T,*}$. Thus, ${\bf u}^*(t,\cdot)={\bf u}(t,\cdot,{\bf u}^{T,*},0)$ is a $T$-periodic solution of \eqref{Eq1}. On the other hand, by the maximum principle for parabolic equations, since ${\bf u^{T,*}}>0$, we have ${\bf u}_{\inf}^*>0$.
	\end{proof}

	\subsection{Proof of Theorem \ref{TH4}}
	
	Throughout this subsection, we shall suppose that {\bf (H1)-(H2)} hold. We first establish the following result.
	\begin{lem}\label{Lem-p-1} Suppose that $\lambda_{*}>0$ and \eqref{Eq1} has a $T$-periodic positive entire solution $\tilde{\bf u}(t,\cdot)$ satisfying {\bf (H2)}. Consider the closed sets 
		\begin{equation}
			\tilde{\mathbb{S}}_{t}:=[C^+(\bar{\Omega})]^2 \cap [ {\bf 0} , \tilde{\bf u}({t},\cdot)], \quad t\in\mathbb{R}.
		\end{equation}
		Then the following conclusions hold.
		\begin{itemize}
			\item[\rm (i)] If ${\bf u}_0\in \tilde{\mathbb{S}}_{t_0}$ for some $t_0\in\mathbb{R}$, then ${\bf u}(t,\cdot;{\bf u}_0,t_0)\in\tilde{\mathbb{S}}_t$ for all $t\ge t_0$. 
			
			\item[\rm(ii)] If $0\le {\bf u}_0<(\tilde{u}_{1,\min},\tilde{u}_{2,\min})$ and ${\bf u}_0\ne 0$, then  $\|{\bf u}(t+t_0,\cdot;{\bf u}_0,t_0)-\tilde{\bf u}(t_0+t,\cdot)\|\to 0$ as $t\to\infty$,   uniformly in $t_0\in\mathbb{R}$. 
		\end{itemize}
		
	\end{lem}
	\begin{proof} Let $\tilde{\bf u}(t,\cdot)$ be a strictly positive entire solution satisfying {\bf (H2)} and set 
		\begin{equation}\label{R-eq5}
			\tilde{h}_1(t,x)=r(t,x)-c(t,x)\tilde{u}_1(t,x)\quad \text{and}\quad \tilde{h}_2(t,x)=s(t,x)-g(t,x)\tilde{u}_2(t,x)\quad x\in\bar{\Omega},\ t\in\mathbb{R}.
		\end{equation}
		Hence $\tilde{h}_1\ge 0$ and $\tilde{h}_2\ge 0$. Now we proceed to prove assertions {\rm (i)} and {\rm (ii)}.
		
		\medskip
		
		{\rm (i)} Thanks to the continuous dependence  of solutions of \eqref{Eq1} with respect to initial data, without loss of generality, we fix $t_0\in\mathbb{R}$ and ${\bf u}_0\in\tilde{\mathbb{S}}_{t_0}$ such that ${\bf u}_0<<\tilde{\bf u}(t_0,\cdot)$. Let $t_{\max}$ be defined as 
		$$
		t_{\max}:=\sup\{t>0 : {\bf u}(\tau+t_0,\cdot;{\bf u}_0;t_0)\in \tilde{\mathbb{S}}_{\tau} \ 0\le \tau\le t\}.
		$$
		Since $\tilde{\bf u}(t_0,\cdot)-{\bf u}_0\in C^{++}(\bar{\Omega})$, then $t_{\max}$ is well defined and $t_{\max}\in (0,\infty]$. We claim that $t_{\max}=\infty$. Suppose to the contrary that $t_{\max}<\infty$. Hence, there is $x^*\in\bar{\Omega}$ such that either
		\begin{equation}\label{R-eq3}
			\tilde{ u}_1(t_0+t_{\max},x^*)-u_2(t_0+t_{\max},x^*;{\bf u}_0,t_0)=0 \quad \text{or}\quad \tilde{ u}_2(t_0+t_{\max},x^*)-u_2(t_0+t_{\max},x^*;{\bf u}_0,t_0)=0. 
		\end{equation}
		Define $\tilde{\bf w}(t,\cdot)=\tilde{\bf u}(t,\cdot)-{\bf u}(t,\cdot;{\bf u}_0,t_0)$, $t\ge t_0$. Then 
		\begin{equation}\label{PP12}
			\begin{cases}
				\partial_t \tilde{w}_1=d_1\Delta \tilde{w}_1+\tilde{h}_1\tilde{w}_2-(a+s+b(\tilde{u}_1+u_1)+cu_2)\tilde{w}_1 & x\in\Omega,\ t>t_0,\cr 
				\partial_t\tilde{w}_2=d_2\Delta\tilde{w}_2+\tilde{h}_2\tilde{w}_1-(e+f(\tilde{u}_2+u_2)+gu_1)\tilde{w}_2 & x\in\Omega,\ t>t_0,\cr
				0=\partial_{\vec{n}}\tilde{w}_1=\partial_{\vec{n}}\tilde{w}_2 & x\in\partial\Omega,\ t>t_0,
			\end{cases}
		\end{equation}
		where $\tilde{h}_1\ge 0$ and $\tilde{h}_2\ge 0$ are given by \eqref{R-eq5}. Observing that $w_1(t,\cdot)\ge 0$ and $w_{2}(t,\cdot)\ge 0$ for $t_0\le t\le t_{\max}$, it follows from \eqref{PP12} that 
		\begin{equation*}
			\begin{cases}
				\partial_t\tilde{w}_1\ge d_1\Delta \tilde{w}_1 -(a+s+b(\tilde{u}_1+u_1)+cu_2)\tilde{w}_1 & x\in\Omega,\ t_0<t\le t_0+t_{\max},\cr 
				0=\partial_{\vec{n}}\tilde{w}_1 & x\in\partial\Omega,\ t_0<t\le t_{\max}+t_{0}
			\end{cases}
		\end{equation*}
		and 
		\begin{equation*}
			\begin{cases}
				\partial_t\tilde{w}_2\ge d_2\Delta \tilde{w}_2 -(e+f(\tilde{u}_1+u_2)+gu_1)\tilde{w}_2 & x\in\Omega,\  t_0<t\le t_0+t_{\max},\cr 
				0=\partial_{\vec{n}}\tilde{w}_2 & x\in\partial\Omega,\ t_0<t\le t_{0}+t_{\max}.
			\end{cases}
		\end{equation*}
		Therefore, since $\tilde{\bf w}(t_0,\cdot)\in C^{++}(\bar{\Omega})$, it follows from the strong maximum principle for parabolic equations that $\tilde{w}_{1}(t_0+t_{\max},\cdot)\in C^{++}(\bar{\Omega})$ and $\tilde{w}_2(t_0+t_{\max},\cdot)\in C^{++}(\bar{\Omega})$, which contradicts \eqref{R-eq3}. Hence we must have that $t_{\max}=\infty$, which yields the desired result.
		
		\medskip
		
		{\rm (ii)} Fix ${\bf 0}<{\bf u}_0<(\tilde{u}_{1,\min},\tilde{u}_{2,\min})$ and ${\bf u}_0\ne 0$. Set 
		$$
		\tilde{\sigma}(t)=\inf\{\sigma>1: \tilde{\bf u}(t+t_0)\le \sigma {\bf u}(t+t_0,\cdot;{\bf u}_0,t_0)\ \forall\ t_0\in\mathbb{R}\},\quad t\ge 0.
		$$
		We first claim that $\tilde{\sigma}(t)$ is nonincreasing in $t\ge 0$. 
		Indeed, fix $t_0\in\mathbb{R}$ and $\tau\ge 0$. Then, since ${\bf u}(t,\cdot;{\bf u}_0,t_0)\le \tilde{\bf u}(t,\cdot)$ for all $t\ge t_0$, we have 
		\begin{align*}
			\partial_t(\tilde{\sigma}(\tau)u_1)=&d_{1}\Delta (\tilde{\sigma}(\tau)u_1)+(r-cu_1)(\tilde{\sigma}(\tau)u_2)-(a+s+bu_1)(\tilde{\sigma}(\tau)u_1)\cr 
			\ge & d_1\Delta (\tilde{\sigma}(\tau)u_1)+(r-c\tilde{u}_1)(\tilde{\sigma}(\tau)u_2)-(a+s+b\tilde{u}_1)(\tilde{\sigma}(\tau)u_1) \quad t> \tau+t_0.
		\end{align*}
		Similarly
		$$
		\partial_t(\tilde{\sigma}(\tau)u_2)=d_2\Delta (\tilde{\sigma}(\tau)u_2)+(s-g\tilde{u}_2)(\tilde{\sigma}(\tau)u_1)-(e+f\tilde{u}_2)(\tilde{\sigma}(\tau)u_2) \quad t\ge \tau+t_0.
		$$
		Hence
		\begin{equation}
			\begin{cases}
				\partial_t(\tilde{\sigma}(\tau)u_1)\ge d_1\Delta (\tilde{\sigma}(\tau)u_1)+\tilde{h}_1(t,\cdot)(\tilde{\sigma}(\tau)u_2)-(a+s+b\tilde{u}_1)(\tilde{\sigma}(\tau)u_1) & x\in\Omega,\ t> \tau+t_0,\cr 
				\partial_t(\tilde{\sigma}(\tau)u_2)\ge d_2\Delta (\tilde{\sigma}(\tau)u_2)+\tilde{h}_2(t,\cdot)(\tilde{\sigma}(\tau)u_1)-(e+f\tilde{u}_2)(\tilde{\sigma}(\tau)u_2)   & x\in\Omega,\ t> \tau+t_0,\cr 
				0=\partial_{\vec{n}}(\tilde{\sigma}(\tau)u_1)=\partial_{\vec{n}}(\tilde{\sigma}(\tau)u_2) & x\in\partial\Omega,\ t>\tau+t_0,\cr 
				\tilde{\sigma}(\tau){\bf u}(\cdot,\cdot;{\bf u}_0,t_0)\ge \tilde{\bf u} & x\in\bar{\Omega},\ t=\tau+t_0,
			\end{cases}
		\end{equation}
		where $\tilde{h}_i\ge 0$, $i=1,2$, is as in \eqref{R-eq5}.
		Hence,  since $\tilde{\bf u}(t,\cdot)$ is a positive entire solution of \eqref{Eq1}, it follows from the comparison principle for cooperative systems that $\tilde{\bf u}(t,\cdot)\le \tilde{\sigma}(\tau){\bf u}(t,\cdot;{\bf u}_0,t_0)$ for any $t\ge t_0+\tau$. Since $t_0$ is arbitrary chosen, we deduce that $\tilde{\sigma}(t)\le \tilde{\sigma}(\tau)$ for every $t\ge \tau\ge 0$. Therefore, 
		\begin{equation}
			\tilde{\sigma}^{\infty}:=\lim_{t\to\infty}\tilde{\sigma}(t)=\inf_{t\ge 0}\tilde{\sigma}(t).
		\end{equation}
		It is clear from the definition of $\tilde{\sigma}(t)$ that $\tilde{\sigma}^{\infty}\ge 1$. Next, we claim that 
		\begin{equation}\label{R-eq4}
			\tilde{\sigma}^{\infty}=1.
		\end{equation}
		Suppose to the contrary that \eqref{R-eq4} is false, that is $\tilde{\sigma}^{\infty}>1$.  Since ${\bf u}_0>{\bf 0}$, by \eqref{TH2-eq1}, there is $\eta^*>0$ such that 
		\begin{equation}\label{R-eq6}
			\eta^*{\bf 1}\le {\bf u}(t+t_0,\cdot;{\bf u}_0,t_0)\quad \forall\ t\ge T,\  t_0\in\mathbb{R}.
		\end{equation}
		Next, choose $0<\delta_*\ll 1$ such that 
		\begin{equation}\label{R-eq7}
			\tilde{\sigma}^{\infty}>e^{\delta_*} \quad \text{and}\quad \min\{b_{\min},f_{\min}\}(\tilde{\sigma}^{\infty}-e^{\delta_*})\eta^*>\delta_*e^{\delta_*}.
		\end{equation}
		Next, fix $\tau\ge T$ and $t_0\in\mathbb{R}$ and define 
		$$
		{\bf v}(t,\cdot)=\tilde{\sigma}(\tau)e^{-\delta_*t}{\bf u}(t+t_0+\tau,\cdot;{\bf u}_0,t_0)\quad 0\le t\le 1. 
		$$
		Recall from {\rm (i)} that $r-c{u}_1\ge r-c\tilde{u}_1=\tilde{h}_{1}$. Hence, thanks to \eqref{R-eq6} and \eqref{R-eq7},  we have 
		\begin{align}\label{PP3}
			\partial_tv_1=&d_1\Delta v_1+(r-cu_1)v_2-(a+s+b u_1(t+\tau+t_0,\cdot))v_1-\delta_*v_1\cr 
			\ge & d_1\Delta v_1+\tilde{h}_1(t,x) v_2-(a+s+b v_1+b(u_1(t+\tau+t_0,\cdot)-v_1))v_1-\delta_*v_1\cr 
			=& d_1\Delta v_1+\tilde{h}_1 v_2-(a+s+b v_1)v_1+(b(v_1-u_1(t+\tau+t_0,\cdot))-\delta_*)v_1\cr 
			=& d_1\Delta v_1+\tilde{h}_1 v_2-(a+s+b v_1)v_1+(b(\tilde{\sigma}(\tau)-e^{\delta_*t})u_1(t+\tau+t_0,\cdot)-\delta_*e^{-\delta_*t})v_1e^{-\delta_*t}\cr 
			\ge & d_1\Delta v_1+\tilde{h}_1 v_2-(a+s+b v_1)v_1+(b(\tilde{\sigma}_{\infty}-e^{\delta_*t})u_1(t+\tau+t_0,\cdot)-\delta_*e^{\delta_*t})e^{-\delta_*t}v_1\cr 
			\ge& d_1\Delta v_1+\tilde{h}_1 v_2-(a+s+b v_1)v_1+(b_{\min}(\tilde{\sigma}_{\infty}-e^{\delta_*t})\eta^*-\delta_*e^{\delta_*t})v_1e^{-\delta_*t}\cr
			\ge& d_1\Delta v_1+\tilde{h}_1 v_2-(a+s+b v_1)v_1\quad 0<t\le 1,\quad x\in\Omega,
		\end{align}
		where  the functions $a,b,s,r,e$ and $f$ are evaluated at $t+\tau+t_0$.
		Similarly, we have 
		\begin{equation}\label{PP4}
			\partial_tv_2\ge d_2\Delta v_2+\tilde{h}_2v_1-(e+fv_2)v_2 \quad 0<t\le 1,\ x\in\Omega.
		\end{equation}
		Observing that $\partial_{\vec{n}}v_i=0$ on $(0,1]\times\partial\Omega$ and 
		$$ 
		{\bf v}(0,\cdot)=\tilde{\sigma}(\tau){\bf u}(\tau+t_0,\cdot;{\bf u}_0,t_0)\ge \tilde{\bf u}(\tau+t_0,\cdot),
		$$
		we can employ the comparison principle for cooperative systems to conclude that 
		\begin{equation*}
			{\bf v}(t,\cdot)\ge \tilde{\bf u}(t+\tau+t_0,\cdot)\quad \forall 0\le t\le 1.
		\end{equation*}
		Taking $t=1$ in the last inequality, we get
		\begin{equation}\label{R-eq8}
			\tilde{\sigma}(\tau)e^{-\delta_*}{\bf u}(1+\tau+t_0,\cdot;{\bf u}_0,t_0)\ge \tilde{\bf u}(\tau+1+t_0,\cdot)\quad \forall\ \tau\ge T, \ t_0\in\mathbb{R}.
		\end{equation}
		By \eqref{R-eq8} and the definition of $\tilde{\sigma}(t)$, we have
		$$
		\tilde{\sigma}(t+1)\le\tilde{\sigma}(\tau)e^{-\delta_*}\quad \forall\ \tau\ge T.
		$$
		Letting $\tau\to\infty$ leads to $\tilde{\sigma}^{\infty}\le \tilde{\sigma}^{\infty}e^{-\delta_*}$, which is impossible since $\delta_*>0$. Therefore, we must have that $\tilde{\sigma}^{\infty}=1$, that is \eqref{R-eq4} holds. 
		
		\medskip
		
		Finally, by \eqref{R-eq4} and the fact that ${\bf u}(t+t_0,\cdot;{\bf u}_0,t_0)\le \tilde{u}(t+t_0,\cdot;{\bf u}_0,t_0)\le \tilde{\sigma}(t){\bf u}(t+t_0,\cdot;{\bf u}_0,t_0)$ for all $t\ge 0$ and $t_0\in\mathbb{R}$, we conclude that $\|{\bf u}(t+t_0,\cdot;{\bf u}_0,t_0)-\tilde{\bf u}(t_0+t,\cdot)\|\to 0$ as $t\to\infty$,   uniformly in $t_0\in\mathbb{R}$.
	\end{proof}
	
	Thanks to Lemma \ref{Lem-p-1}, we can give a proof of Theorem \ref{TH4}. 
	
	\begin{proof}[Proof of Theorem \ref{TH4}.] Assume that the hypothesis of the theorem holds. Let ${\bf u}_0\in [C^{++}(\bar{\Omega})]^2$. Next, set $\underline{\bf u}_0=(\frac{\min\{\tilde{u}_{1,\min},u_{0,1}\}}{2},\frac{\min\{\tilde{u}_{2,\min},u_{0,2}\}}{2})$.  Then by Lemma \ref{Lem-p-1}, setting $\underline{\bf u}(t,\cdot;t_0)={\bf u}(t,\cdot;\underline{\bf u}_0,t_0)$ for  $t\ge t_0\in\mathbb{R}$, it holds that 
		\begin{equation}\label{PP7}
			\underline{\bf u}(t,\cdot;t_0)\in\tilde{\mathbb{S}}_{t}, \ \forall\ t\ge t_0\in\mathbb{R} \quad\text{and}\quad \lim_{t\to\infty}\sup_{t_0\in \mathbb{R}}\|\underline{\bf u}(t+t_0,\cdot;t_0)-\tilde{\bf u}(t_0+t,\cdot)\|=0.
		\end{equation}
		From this point, the proof is divided into two steps. 
		
		\noindent{\bf Step 1.} Here we show that \begin{equation}\label{PP8}
			\underline{\bf u}(t,\cdot;t_0)\le {\bf u}(t,\cdot;{\bf u}_0,t_0)\quad\ t\ge t_0.
		\end{equation} 
		Fix $t_0\in\mathbb{R}$ and set ${\bf w}(t,\cdot;t_0)={\bf u}(t,\cdot;{\bf u}_0,t_0)-\underline{\bf u}(t,\cdot;t_0)$, $t\ge t_0$. Then
		\begin{align}\label{BB1-2}
			\partial_tw_1=&d_1\Delta w_1+(r-cu_{1})u_{2}-(r-c\underline{u}_{1})\underline{u}_{2}- (a+s+b(u_{1}+\underline{u}_1))w_1 & t>t_0,\cr 
			=& d_2\Delta w_1+(r-c\underline{u}_{1})w_2-(a+s+b(u_{1}+\underline{u}_{1})+cu_{2})w_1 & t>t_0.
		\end{align}
		Similarly, 
		\begin{equation}\label{BB2-2}
			\partial_t w_2=d_2\Delta w_2 +(s-g\underline{u}_{2})w_1 -(e+f(u_{2}+\underline{u}_{2})+gu_{1})w_2\quad t>t_0.
		\end{equation}
		Therefore,   setting 
		$$
		h_{1,2}=r-c\underline{u}_{1}, \quad h_{1,1}=a+s+b(u_{1}+\underline{u}_{1})+cu_{2},
		$$
		$$
		h_{2,1}=s-g\underline{u}, \quad \text{and}\quad h_{2,2}=e+f(u_{2}+\underline{u}_{2})+gu_{1},
		$$
		then,  by the results of Lemma \ref{Lem-p-1},  $h_{i,j}>0$ for $i,j=1,2$. Thus, by \eqref{BB1-2} and \eqref{BB2-2}, ${\bf w}(t,x)$ solves the cooperative system
		\begin{equation}\label{BB4-2}
			\begin{cases}
				\partial_tw_1=d_1\Delta w_1 -h_{1,1}w_1 +h_{1,2}w_2 & x\in\Omega,\ t>t_0,\cr 
				\partial_tw_2=d_2\Delta w_2+h_{1,2}w_1 -h_{2,2}w_2 & x\in\Omega,\ t>t_0,\cr 
				0=\partial_{\vec{n}}w_1=\partial_{\vec{n}}w_2 & x\in\partial\Omega,\ t>t_0.
			\end{cases}
		\end{equation}
		Since $w_1(t_0,\cdot;t_0)>0$ and $w_2(t_0,\cdot;t_0)>0$, it follows from the comparison principle for cooperative systems that $w_{1}(t+t_0,x;t_0)>0$ and $w_{2}(t+t_0,x;t_0)>0$ for $t>0$ and $x\in\mathbb{R}^n$, which yields \eqref{PP8}.
		
		\noindent{\bf Step 2.} We complete the proof here by establishing that 
		\begin{equation}\label{PP9}
			\lim_{t\to\infty}\sup_{t_0\in\mathbb{R}}\|{\bf u}(t+t_0,\cdot;{\bf u}_0,t_0)-\tilde{\bf u}(t_0+t,\cdot)\|=0.
		\end{equation}
		Thanks to \eqref{PP7}, to obtain that \eqref{PP9} holds, it is enough to establish that 
		\begin{equation}\label{PP10}
			\lim_{t\to\infty}\sup_{t_0\in\mathbb{R}}\|{\bf w}({t_0}+t,\cdot;t_0)\|=0
		\end{equation}
		where ${\bf w}(t+t_0,\cdot;t_0)$ $t\ge 0$, $t_0\in\mathbb{R}$ is as in Step 1. We proceed by contradiction to prove the validity of \eqref{PP10}.  Hence, we suppose that there is a sequence $\{t_n\}_{n\ge 1}$ of positive numbers converging to infinity and a sequence of initial times $\{t_{0,n}\}_{n\ge 1}$ such that 
		\begin{equation}\label{PP11}
			\inf_{n\ge 1}\|{\bf w}(t_{0,n}+t_n,\cdot;t_{0,n})\|>0.
		\end{equation}
		Consider the sequence of  ${\bf w}^n(t,\cdot):={\bf w}(t_{0,n}+t_n+t,\cdot;t_{0,n})$, $t>-t_n$, $n\ge 1$. Thanks to the regularity theory for parabolic equations and the fact that solutions of \eqref{Eq1} are eventually bounded, then without loss of generality, we may suppose that ${\bf w}^n(t,\cdot)\to {\bf w}^{\infty}(t,\cdot)$ and ${\bf u}(t+t_n+t_{0,n},\cdot;t_{0,n})\to {\bf u}^{\infty}(t,\cdot)$ as $n\to\infty$ locally in  $C^{1,2}(\mathbb{R}\times\bar{\Omega})$. Furthermore, by {\bf (H1)}, we may also suppose that $\tau(t+t_n+t_{0,n},\cdot)\to\tau^{\infty}(t,\cdot)$ (a $T$-periodic function) as $n\to\infty$  locally uniformly in $\mathbb{R}\times\bar{\Omega}$ for each $\tau\in\{a,b,c,r,s,e,f,g,\tilde{u}_1,\tilde{u}_2\}$, and  thanks to \eqref{BB4-2}$, {\bf w}^{\infty}$ satisfies  
		\begin{equation}\label{BB4-3}
			\begin{cases}
				\partial_t{w}_1^{\infty}=d_1\Delta {w}_1^{\infty} -h_{1,1}^{\infty}{w}_1^{\infty} +h_{1,2}^{\infty}{w}_2^{\infty} & x\in\Omega,\ t\in \mathbb{R},\cr 
				\partial_tw_2^{\infty}=d_2\Delta w_2^{\infty}+h_{1,2}^{\infty}w_1^{\infty} -h_{2,2}^{\infty}w_2^{\infty} & x\in\Omega,\ t\in\mathbb{R},\cr 
				0=\partial_{\vec{n}}w_1^{\infty}=\partial_{\vec{n}}w_2^{\infty} & x\in \partial \Omega,\ t\in\mathbb{R},
			\end{cases}
		\end{equation}
		where $h_{1,2}^{\infty}=r^{\infty}-c\tilde{u}_1^{\infty}$, $h_{1,1}^{\infty}=a^{\infty}+s^{\infty}+b^{\infty}(u_1^{\infty}+\tilde{u}^{\infty}_1)+c^{\infty}u^{\infty}_2$, $h_{2,1}^{\infty}=s^{\infty}-g^{\infty}\tilde{u}_1^{\infty}$, and $h_{2,2}^{\infty}=e^{\infty}+f^{\infty}(u_2^{\infty}+\tilde{u}^{\infty}_2)+g^{\infty}u^{\infty}_1$. Clearly  $h_{1,2}^{\infty}$ and $h_{2,2}^{\infty}>0$. It also follows from \eqref{PP7} and Step 1 that $h_{1,1}^{\infty}\ge 0$, $h_{2,1}^{\infty}\ge 0$, $u_1^{\infty}\ge \tilde{u}^{\infty}_1$ and $u^{\infty}_2\ge \tilde{u}^{\infty}_2$. Therefore, taking $\sigma_*=\min\{b_{\min}\tilde{u}_{1,\min},f_{\min}\tilde{u}_{2,\min}\}>0$, we deduce from \eqref{BB4-3} that 
		\begin{equation}\label{BB4-4}
			\begin{cases}
				\sigma_*w_1^{\infty}+  \partial_t{w}_1^{\infty}\le d_1\Delta {w}_1^{\infty} -\underbrace{(a^{\infty}+s^{\infty}+b^{\infty}\tilde{u}_1^{\infty})}_{\tilde{h}_{1,1}^{\infty}}{w}_1^{\infty} +h_{1,2}^{\infty}{w}_2^{\infty} & x\in\Omega,\ t\in \mathbb{R},\cr 
				\sigma_*w_2^{\infty}+   \partial_tw_2^{\infty}\le d_2\Delta w_2^{\infty} -\underbrace{(e^{\infty}+f^{\infty}\tilde{u}_2^{\infty})}_{\tilde{h}_{2,2}^{\infty}}w_2^{\infty} +h_{1,2}^{\infty}w_1^{\infty}& x\in\Omega,\ t\in\mathbb{R},\cr 
				0=\partial_{\vec{n}}w_1^{\infty}=\partial_{\vec{n}}w_2^{\infty} & x\in \partial\Omega,\ t\in\mathbb{R}.
			\end{cases}
		\end{equation}
		Now, note that since $\tilde{\bf u}(t,\cdot)$ is a $T$-periodic strictly positive entire solution of \eqref{Eq1} and $\tilde{\bf u}(t_{0,n}+t_n+t,\cdot)\to \tilde{\bf u}^{\infty}(t,\cdot)$ as $n\to\infty$ locally uniformly in $\mathbb{R}\times{\bar{\Omega}}$, then $\tilde{\bf u}^{\infty}$ is a strictly positive entire solution of 
		\begin{equation}\label{BB4-5}
			\begin{cases}  \partial_t\tilde{u}^{\infty}_1=  d_1\Delta \tilde{u}^{\infty}_1 -\tilde{h}_{1,1}^{\infty}\tilde{u}^{\infty}_1+h_{1,2}^{\infty}\tilde{u}^{\infty}_2 & x\in\Omega,\ t\in \mathbb{R},\cr 
				\partial_t\tilde{u}^{\infty}_2= d_2\Delta \tilde{u}_2^{\infty} -\tilde{h}_{2,2}^{\infty}u_2^{\infty} +h_{1,2}^{\infty}\tilde{u}^{\infty}_1& x\in\Omega,\ t\in\mathbb{R},\cr      0=\partial_{\vec{n}}\tilde{u}^{\infty}_1=\partial_{\vec{n}}\tilde{u}^{\infty}_2 & x\in \partial\Omega,\ t\in\mathbb{R}.
			\end{cases}
		\end{equation}
		Recall that $M^*:=\sup_{t\in\mathbb{R}}\|{\bf w}(t,\cdot)\|\le \sup_{t\ge t_0}\|{\bf u}(t,\cdot;t_0)\|<\infty$ (by Lemma \ref{Lem2-5}). Setting  
		$$\underline{\bf v}(t,\cdot,t_0):=e^{\sigma_*(t-t_0)}{\bf w}^{\infty}(t,\cdot)\quad \text{and}\quad  \overline{\bf v}(t,\cdot;t_0)=\frac{M^*}{\min_{i=1,2}\{\tilde{u}_{i,\min}\}}\tilde{\bf u}^{\infty}(t,\cdot),\quad t\ge t_0,\ t,, t_0\in\mathbb{R},$$
		we have that $\underline{\bf v}(t_0,\cdot,t_0)\le \overline{\bf v}(t_0,\cdot;t_0)$ for all $t_0\in\mathbb{R}$. Observe $\overline{\bf v}$  solves the linear cooperative system \eqref{BB4-5}, while thanks to \eqref{BB4-4}, $\underline{\bf v}(t,\cdot;t_0)$ is a subsolution of \eqref{BB4-5}. Therefore, by the comparison principle for cooperative systems, we have that $\underline{\bf v}(t,\cdot;t_0)\le \overline{\bf v }(t,\cdot;t_0)$ for all $t\ge t_0$. Hence
		$$
		\|{\bf w}^{\infty}(t,\cdot)\|\le \frac{M^*e^{-\sigma_*(t-t_0)}}{\min_{i=1,2}\tilde{u}_{i,\min}}\|\tilde{\bf u}^{\infty}(t,\cdot)\|\le \frac{M^*\max_{i=1,2}\tilde{u}_{i,\max}}{\min_{i=1,2}\tilde{u}_{i,\min}}e^{-\sigma_{*}(t-t_0)}\quad \forall\ t\ge t_0.
		$$
		Letting $t_0\to-\infty$ in the last inequality gives $\|{\bf w}^{\infty}(t,\cdot)\|=0$ for all $t\in\mathbb{R}$, which implies that $\|{\bf w}(t_{0,n}+t_n,\cdot;t_{0,n})\|\to 0$ as $n\to\infty$. This clearly contradicts with \eqref{PP11}. Therefore, \eqref{PP10} holds, which completes the proof of the theorem.  
	\end{proof}
	
	\begin{proof}[Proof of Proposition \ref{Cor2}]  Let $(\tilde{r}(t,\cdot),\tilde{s}(t,\cdot))$ be given as in Corollary \eqref{Cor2}. We proceed in two steps.
		
		{\bf Step 1}.  Fix ${\bf u}_0\in [C^+(\bar{\Omega})]^2$ and $t_0\in\mathbb{R}$ such that ${\bf u}_0\le (\tilde{r}(t_0,\cdot),\tilde{s}(t_0,\cdot))$. We claim that 
		\begin{equation}\label{R-eq9}
			{\bf u}(t,\cdot;{\bf u}_0,t_0)\le(\tilde{r}(t,\cdot),\tilde{s}(t,\cdot))\quad \forall\ t\ge t_0.
		\end{equation}
		To this end, observe that $u_1(t,\cdot):=u_1(t,\cdot;{\bf u}_0,t_0)$ satisfies 
		$$
		\partial_tu_1-d_1\Delta u_1-c(t,x)(\tilde{r}-u_1)u_2+(a+s+bu_1)u_1=0 \quad x\in\Omega,\ t>t_0; \quad \partial_{\vec{n}}u_1=0\quad x\in\partial\Omega,\ t>t_0,
		$$
		and by \eqref{Rk1-eq1}, $\tilde{r}$ satisfies 
		\begin{align*}
			\partial_t\tilde{r}-d_1\Delta\tilde{r}- c(t,x)(\tilde{r}-\tilde{r})u_2+(a+s+b\tilde{r})\tilde{r}\ge 0\quad x\in\Omega,\ t\in\mathbb{R}; \quad \partial_{\vec{n}}\tilde{r}\ge 0\quad x\in\partial\Omega,\ t\in\mathbb{R}.  
		\end{align*}
		Hence, by the comparison principle for  parabolic equations, we conclude that $u_1(t,\cdot)\le \tilde{r}(t,\cdot)$ for all $t\ge t_0$. Similarly, $u_2(t,\cdot):=u_2(t,\cdot;{\bf u}_0,t_0)$ satisfies 
		$$
		\partial_tu_2-d_2\Delta u_2-g(t,x)(\tilde{s}-u_2)u_1+(e+fu_2)u_2=0 \quad x\in\Omega,\ t>t_0; \quad \partial_{\vec{n}}u_1=0\quad x\in\partial\Omega,\ t>t_0,
		$$
		and by \eqref{Rk1-eq1}, $\tilde{s}$ satisfies 
		\begin{align*}
			\partial_t\tilde{s}-d_2\Delta\tilde{s}- g(t,x)(\tilde{s}-\tilde{s})u_1+(e+f\tilde{s})\tilde{s}\ge 0\quad x\in\Omega,\ t\in\mathbb{R}; \quad \partial_{\vec{n}}\tilde{s}\ge 0\quad x\in\partial\Omega,\ t\in\mathbb{R}.  
		\end{align*}
		Hence, by the comparison principle for parabolic equations, we conclude that $u_2(t,\cdot)\le \tilde{s}(t,\cdot)$ for all $t\ge t_0$. Therefore, \eqref{R-eq9} holds.
		
		\medskip

		{\bf Step 2.} Suppose that $\lambda_*>0$. We show that \eqref{Eq1} has a strictly positive $T$-periodic solution $\tilde{\bf u}(t,\cdot)$ satisfying 
		\begin{equation}\label{R-eq10}
			\tilde{\bf u}(t,\cdot)\le (\tilde{r}(t,\cdot),\tilde{s}(t,\cdot))\quad t\in\mathbb{R}.
		\end{equation}
		Fix $0<\alpha<1$ and $\gamma^*$ as in Lemma \ref{Lem2-5}, and $K=K_{T,\alpha,\gamma^*}$ as in Lemma \ref{Lem2-6}. Next, let $\xi^*>0$ be the positive number in Lemma \ref{Lem2-7} that satisfies
		\begin{equation}\label{R-eq11}
			{\bf u}(t_0+T,\cdot;{\bf u}_0,t_0)\in\mathcal{M}^{t_0,\alpha}_{\gamma^*,\xi}\quad \text{whenever}\ {\bf u}_0\in\mathcal{M}^{t_0,\alpha}_{\gamma^*,\xi}\quad \text{and}\ 0<\xi<\xi^*, \ t_0\in\mathbb{R},
		\end{equation}
		where $\mathcal{M}^{t_0,\alpha}_{\gamma^*,\xi}$ is defined by \eqref{RE:18}. Next, chose $0<\tilde{\xi}^*\ll \min\{\xi^*,1\}$ such that 
		$$
		\tilde{\xi}^*\varphi(0,\cdot)<(\tilde{r}(0),\tilde{s}(0,\cdot)).
		$$
		Thanks to Step 1 and the fact that $(\tilde{r}(T,\cdot),\tilde{s}(T,\cdot))=(\tilde{r}(0,\cdot),\tilde{s}(0,\cdot))$, we have that ${\bf u}(T+t_0,\cdot,\cdot;{\bf u}_0,t_0)\in [{\bf 0}, (\tilde{r}(0,\cdot),\tilde{s}(0,\cdot))]$ whenever ${\bf 0}\le {\bf u}_0\le (\tilde{r}(0,\cdot),\tilde{s}(0,\cdot))$.  This along with \eqref{R-eq11} shows that, 
		\begin{equation}
			{\bf u}(T,\cdot;{\bf u}_0,0)\in\mathcal{M}^*:=[{\bf 0},(\tilde{r}(0),\tilde{s}(0))]\cap\mathcal{M}^{0,\alpha}_{\gamma^*,\xi^*}\quad \text{whenever}\quad {\bf u}_0\in\mathcal{M}^*.
		\end{equation} 
		
		Since  $\mathcal{M}^*$ is a closed, convex and bounded subset of $[X^{\alpha}]^2$, and by the arguments in proof of Theorem \ref{TH3}-{\rm (i)}, the Poincar\'e map $\mathcal{M}^*\ni {\bf u}_0\mapsto {\bf u}(T,\cdot;{\bf u}_0,0)\in\mathcal{M}^*$ is compact, then by the Schauder fixed point theorem, there is $\tilde{\bf u}_0\in\mathcal{M}^*$ such that ${\bf u}(T,\cdot;\tilde{\bf u}_0,0)=\tilde{\bf u}_0$. Therefore, $\tilde{\bf u}(t,x)={\bf u}(t,\cdot;\tilde{\bf u}_0,0)$ is a strictly positive and bounded $T$-periodic solution of \eqref{Eq1}. Furthermore, since $\tilde{\bf u}_0\le (\tilde{r}(0,\cdot),\tilde{s}(0,\cdot))$, we deduce from step 1 that $\tilde{\bf u}(t,\cdot)\le (\tilde{r}(t,\cdot),\tilde{s}(t,\cdot))$ for all $t\ge 0$. This along with the periodicity of the maps $\tilde{\bf u}(t,\cdot)$ and $(\tilde{r}(t,\cdot),\tilde{s}(t,\cdot))$ implies that \eqref{R-eq10} holds.  Consequently, $\tilde{\bf u}(t,\cdot)$ satisfies the requirement of hypothesis {\bf (H2)} since $\tilde{r}=\frac{r}{c}$ and $\tilde{s}=\frac{s}{g}$. 
	\end{proof}

	We complete this subsection with a proof of Corollary \ref{Cor1}. 
	
	\begin{proof}[Proof of Corollary \ref{Cor1}] Suppose that {\bf (H1)} holds, $\lambda_*>0$, and $\tilde{c}:=\frac{c}{r}$ and $\tilde{g}=\frac{g}{s}$ are constant functions. If $\tilde{c}=\tilde{g}=0$, then any positive $T$-periodic positive solution of \eqref{Eq1} satisfies {\bf (H2)}, and hence the result follows from Theorem \ref{TH4}. If $\tilde{c}>0$ and $\tilde{g}>0$, then taking $\tilde{r}=\frac{1}{\tilde{c}}$ and $\tilde{s}=\frac{1}{\tilde{g}}$, we have that \eqref{Rk1-eq1} of Proposition \ref{Cor2} holds, and hence the result follows again.  Hence, it remains the case of $\tilde{c}\tilde{g}=0$ and $\tilde{c}+\tilde{g}>0$.  Without loss of generality, with suppose that $\tilde{c}=0$ and $\tilde{g}>0$. By Theorem \ref{TH3}, we know that \eqref{Eq1} has at least one $T$-periodic positive solution $\tilde{\bf u}$. Note that $\tilde{u}_2$ satisfies
		$$
		\begin{cases}
			\partial_t\tilde{u}_2\le d_2\Delta \tilde{u}_2+s(1-\tilde{g}\tilde{u}_2)\tilde{u}_1 & x\in\Omega,\ t\in \mathbb{R},\cr
			0=\partial_{\vec{n}}\tilde{u}_2 & x\in\partial\Omega,\ t\in\mathbb{R}.
		\end{cases}
		$$
		Hence, since $\tilde{u}_1,\tilde{u}_2\in \mathcal{X}_T^{++}$, we can employ standard stability results on the single species Logistic-type equations and the comparison principle for parabolic equations to deduce that $\tilde{g}\tilde{u}_2\le 1$, which implies that $s-g\tilde{u}_2=s(1-\tilde{g}\tilde{u}_2)\ge 0$. It is also clear that $r-c\tilde{u}_1=r(1-\tilde{c}\tilde{u}_1)\ge 0$. Hence $\tilde{\bf u}$ satisfies {\bf (H2)}, and hence the result follows from Theorem \ref{TH4}. 
	\end{proof}
	
	\section{Proof of  Theorem \ref{TH5}}
	
	Throughout the rest of this paper, we shall always suppose that {\bf (H1)} and ${\bf (H3)}$ hold. We establish some lemmas. For every $\varepsilon\ge 0$, define
	\begin{equation}\label{G-funct-def}
		G^{\varepsilon}(x,\tau)=\sqrt{(a(x)+s(x)+\tau c(x))^2+4\tau b(x)(\varepsilon+r(x))}+(a(x)+s(x)+\tau c(x))\quad x\in\bar{\Omega},\ \tau\ge0,
	\end{equation}
	and \begin{equation}\label{w-1-def}
		w_1^{\varepsilon}(x;\tau)=\frac{2(\varepsilon+r(x))\tau}{G^{\varepsilon}(x,\tau)}\quad \forall\ x\in\bar{\Omega},\,\ \tau\ge 0.
	\end{equation}
	
	\begin{lem}\label{l1}
		For every $\varepsilon\ge 0$, let  $G^{\varepsilon}$ and $w_1^{\varepsilon}$ be defined as in \eqref{G-funct-def} and \eqref{w-1-def}, respectively.
		\begin{itemize}
			\item[\rm (i)] $G^{\varepsilon}$ is H\"older continuous, and for each $x\in\bar{\Omega}$, $G^{\varepsilon}(x,\tau)$ is strictly increasing in $\tau\ge 0$. Moreover, for every $\tau\ge0 $, it holds that  $\min_{y\in\bar{\Omega}}G^{\varepsilon}(y,\tau)\ge 2(a+s)_{\min}>0$. 
			\item[\rm (ii)] $w^{\varepsilon}_1$ is H\"older continuous, and for each $x\in\bar{\Omega}$, $w^{\varepsilon}_1(x;\tau)$ is strictly increasing in $\tau\ge 0$. Moreover
			\begin{equation}\label{l1-eq1}
				-\varepsilon \tau=(r(x)-c(x)w_1^{\varepsilon}(x;\tau))\tau-(a(x)+s(x)+b(x)w_1^{\varepsilon}(x;\tau))w_1^{\varepsilon}(x;\tau)\quad \forall\ x\in\bar{\Omega},\ \tau\ge 0,
			\end{equation}
			
			\item[\rm (iii)] For every $\tau_2>0$, there is $m_{\tau_2}>0$ and $\varepsilon_{\tau_2}>0$ such that
			\begin{equation}\label{l1-eq2}
				(r-cw_1^{\varepsilon}(\cdot;\tau))_{\min}\ge m_{\tau_2}\quad \forall\ 0\le \varepsilon\le \varepsilon_{\tau_2}
				\ \text{and}\ 0\le \tau\le \tau_2.        \end{equation}
		\end{itemize}
	\end{lem}
	\begin{proof}{\rm (i)} The regularity of $G^{\varepsilon}$ in $x$ and $\tau\ge 0$ easily follows from that of the functions $a, b, c, s$ and $r$. It is easy to see from the definition of $G^{\varepsilon}$ that it is strictly increasing in $\tau\ge$. Hence $G^{\varepsilon}(x,\tau)\ge G^{\varepsilon}(x,0)=2(a(x)+s(x))\ge 2(s+a)_{\min}$ for all $x\in\bar{\Omega}$.
		
		\medskip
		
		{\rm (ii)} The regularity of $w_1^{\varepsilon}$ in $x$ and $\tau$ follows from that of $G^{\varepsilon}$ and $r$. It is clear that $w_1^{\varepsilon}(x;\tau)\ge 0$ for all $x\in\bar{\Omega}$ and $\tau\ge 0$, with strict inequality if $\tau>0$. It is clear that \eqref{l1-eq1} holds if $\tau=0$. Now fix $\tau>0$ and $x\in\bar{\Omega}$. Hence $w_1^{\varepsilon}(x;\tau)>0$. Consider the quadratic equation 
		\begin{equation}\label{A1}
			0=(r(x)+\varepsilon)\tau -(a(x)+s(x)+\tau c(x))z -b(x)z^2,
		\end{equation}
		where the unknown is $z$. It follows from the quadratic formula that the unique positive solution of \eqref{A1} is 
		$$
		z^{+}=\frac{(a(x)+s(x)+\tau c(x))-\sqrt{(a(x)+s(x)+\tau c(x))^2+4\tau(\varepsilon+r(x))b(x)}}{-2b(x)}=\frac{2\tau(\varepsilon+r(x))}{G^{\varepsilon}(x,\tau)}.
		$$
		Noting from the formula of $w_1^{\varepsilon}(x;\tau)$ that $z^+=w_1^{\varepsilon}(x;\tau)$, then \eqref{l1-eq1} holds since $z^+$ solves \eqref{A1}.  Thanks to the formula of $z^+$, we have that 
		\begin{align*}
			\partial_{\tau}w_1^{\varepsilon}(x;\tau)=&\frac{c(x)(A(x,\tau)+B(x))-c(x)\sqrt{A^2(x,\tau)+2\tau B(x)}}{2b(x)\sqrt{A^2(x,\tau)+2\tau B(x)}},
		\end{align*}
		where 
		$$ 
		A(x):=a(x)+s(x)+\tau c(x)\quad \text{and}\quad B(x)=2(\varepsilon+r(x)b(x))).
		$$
		Observing that 
		\begin{align*}
			(c(x)A(x,\tau)+B(x))^2=&c^2(x)A^2(x,\tau)+2c(x)A(x,\tau)B(x)+B^2(x)\cr
			=& c^2(x)(A^2(x,\tau)+2\tau  B(x))+2c(x)(a(x)+s(x))B(x)+B^2(x)\cr 
			>&c^2(x)(A^2(x,\tau)+2\tau  B(x)),
		\end{align*}
		then $\partial_{\tau}w_1^{\varepsilon}(x;\tau)>0$ for $x\in\bar{\Omega}$ and $\tau\ge 0$.

		\medskip
		
		{\rm (iii)} When $\tau>0$, it follows from \eqref{l1-eq1} that 
		\begin{align}\label{A2}
			r(x)-c(x)w_1^{\varepsilon}(x;\tau)=&\frac{(a(x)+s(x)+b(x)w_1^{\varepsilon}(x;\tau))w_1^{\varepsilon}(x;\tau)}{\tau}-\varepsilon\cr
			=&\frac{2(a(x)+s(x)+b(x)w_1^{\varepsilon}(x;\tau))(\varepsilon+r(x))}{G^{\varepsilon}(x;\tau)}-\varepsilon
		\end{align}
		Note also that when $\tau=0$, we have $w_1^*(x,0)=0$, $G^{\varepsilon}(x,0)=2(a(x)+s(x))$, and 
		$$
		\frac{2(a(x)+s(x)+b(x)w_1^{\varepsilon}(x;\tau))(\varepsilon+r(x))}{G^{\varepsilon}(x;\tau)}-\varepsilon=r(x)=r(x)-c(x)w_1^{\varepsilon}(x;0).
		$$
		Therefore, \eqref{A2} also holds for $\tau=0$.  Hence, since $G^{\varepsilon}(\cdot;\tau_1)\le G^{1}(\cdot,\tau_1)\le G(\cdot,\tau_2)$, then by \eqref{A2}, for every $0<\tau_1<\tau_2>0$ and $0<\varepsilon<1$, it holds that
		$$
		r(x)-c(x)w_1^{\varepsilon}(x;\tau)\ge \frac{2(a+s)_{\min}(\varepsilon+r_{\min})}{\|G^{1}(\cdot;\tau_2)\|_{\infty}}-\varepsilon,\quad 0<\varepsilon<1, \quad 0\le \tau<\tau_2.
		$$
		Hence, for each fixed $\tau_2>0$, since as $\varepsilon\to0^+$, it holds that  $\frac{2(a+s)_{\min}(\varepsilon+r_{\min})}{\|G^{1}(\cdot;\tau_2)\|_{\infty}}-\varepsilon\to \frac{2(a+s)_{\min}(r_{\min})}{\|G^{1}(\cdot;\tau_2)\|_{\infty}}>0 $, there is $0<\varepsilon_{\tau_{2}}<1$  such that $(r-cw^{\varepsilon}_1(\cdot;\tau))_{\min}\ge \frac{2(a+s)_{\min}(r_{\min})}{2\|G^{1}(\cdot;\tau_2)\|_{\infty}}$ for all $0<\varepsilon<\varepsilon_{\tau_2}$ and $0\le \tau\le \tau_{2}$.
	\end{proof}
	Next, for each $\varepsilon\ge 0$, we introduce the function 
	\begin{equation}\label{h-def-eq}
		h^{\varepsilon}(x)=\frac{(\varepsilon+r(x))s(x)}{a(x)+s(x)}-e(x)\quad x\in\bar{\Omega}.
	\end{equation}
	and 
	\begin{equation}\label{tilde-F-funct-def}
		F^{\varepsilon}(x,\tau)=h_+^{\varepsilon}(x)+\varepsilon-\left(2(\varepsilon+r(x))s(x)\Big(\frac{1}{G^{\varepsilon}(x,0)} -\frac{1}{G^{\varepsilon}(x,\tau)}\Big)+\Big(g(x)w_1^{\varepsilon}(x;\tau)+\tau f(x)\Big)\right)\quad x\in\bar{\Omega},\ \tau\ge 0,
	\end{equation}
	where $h^{\varepsilon}_+=\max\{h^{\varepsilon},0\}$, and  $G^{\varepsilon}$ and $w_1^{\varepsilon}$ are defined by \eqref{G-funct-def} and \eqref{w-1-def}, respectively.

	\begin{lem}\label{l2} Set $\tau_2:=(2\|(2s+rs)/(s+a)\|_{\infty}+\|e\|_{\infty}+2)/f_{\min}$. For every $\varepsilon\ge 0$, let $F^{\varepsilon}$ be defined by \eqref{tilde-F-funct-def}.
		\begin{itemize}
			\item[\rm (i)]$F^{\varepsilon}$ is H\"older continuous, strictly decreasing in $\tau\ge 0$; $F^{\varepsilon}(x,0)\ge \varepsilon$ for all $x\in\bar{\Omega}$, and  $F^{\varepsilon}(x,\tau_{2})<0$ for all $x\in\bar{\Omega}$ and $0\le \varepsilon\le 1$.
			\item[\rm (ii)] For every $0< \varepsilon<1$ and $x\in\bar{\Omega}$, there is a unique positive number $0<w_{2}^{\varepsilon}(x)<$ such that $F^{\varepsilon}(x,w_2^{\varepsilon}(x))=0$. Moreover, for every $0<\varepsilon\le 1$, $w_2^{\varepsilon}$ is H\"older continuous in $x\in\bar{\Omega}$. Furthermore,  
			\begin{equation}\label{l2-eq1}
				-\varepsilon w_2^{\varepsilon}(x)\ge (s(x)-g(x)w_2^{\varepsilon}(x))w_1^{\varepsilon}(x;w_2^{\varepsilon}(x))-(e(x)+f(x)w_2^{\varepsilon}(x))w_2^{\varepsilon}(x)\quad 0<\varepsilon<1, \ x\in\bar{\Omega},
			\end{equation}
			where $w_1^{\varepsilon}$ is as in \eqref{w-1-def}.
			
			\item[\rm (iii)] If in addition  $e_{\min}>0$,   then there is $\tilde{\varepsilon}_{\tau_2}>0$  and $\tilde{m}_{\tau_2}>0$ such that 
			\begin{equation}\label{l2-eq2}
				(s-gw_2^{\varepsilon})_{\min}\ge \tilde{m}_{\tau_2} \quad \forall\ 0< \varepsilon<\tilde{\varepsilon}_{\tau_2}.
			\end{equation}
			
		\end{itemize}

	\end{lem}
	\begin{proof} {\rm (i)} It is clear that {\rm (i)} readily follows from \eqref{tilde-F-funct-def}.

		\medskip
		
		{\rm(ii)} Since for every $x\in\bar{\Omega}$ and $0<\varepsilon\le 1$, $F^{\varepsilon}(x,0)>0$ and  $F^{\varepsilon}(x,\tau_2)<0$, then by the intermediate value theorem there is $w_2^{\varepsilon}(x)$ such that $F^{\varepsilon}(x,w_2^{\varepsilon}(x))=0$. Observing that $\partial_{\tau}F^{\varepsilon}(x,\tau)<0$ for all $\tau\ge 0$, the the uniqueness and regularity of $w_2^{\varepsilon}$ follows from the implicit function theorem. Finally, fix $0<\varepsilon\le 1$ and $x\in\bar{\Omega}$. Then,
		\begin{align*}
			0=&w_2^{\varepsilon}(x)F^{\varepsilon}(x,w_2^{\varepsilon}(x))\cr 
			= & (h^{\varepsilon}_+(x)+\varepsilon)w_2^{\varepsilon}(x)-\frac{2(\varepsilon+r(x))s(x)w_2^{\varepsilon}(x))}{G^{\varepsilon}(x,0)}+(s(x)-g(x)w_2^{\varepsilon}(x))w_1^{\varepsilon}(x;w_2^{\varepsilon}(x))-f(x)[w_2^{\varepsilon}(x)]^2\cr 
			\ge& (h^{\varepsilon}(x)+\varepsilon)w_2^{\varepsilon}(x)-\frac{2(\varepsilon+r(x))s(x)w_2^{\varepsilon}(x))}{G^{\varepsilon}(x,0)}+(s(x)-g(x)w_2^{\varepsilon}(x))w_1^{\varepsilon}(x;w_2^{\varepsilon}(x))-f(x)[w_2^{\varepsilon}(x)]^2\cr
			=& (h^{\varepsilon}(x)+\varepsilon)w_2^{\varepsilon}(x)-\frac{(\varepsilon+r(x))s(x)w_2^{\varepsilon}(x))}{a(x)+s(x)}+(s(x)-g(x)w_2^{\varepsilon}(x))w_1^{\varepsilon}(x;w_2^{\varepsilon}(x))-f(x)[w_2^{\varepsilon}(x)]^2\cr
			=& \varepsilon w_2^{\varepsilon}(x)+(s(x)-g(x)w_2^{\varepsilon}(x))w_1^{\varepsilon}(x;w_2^{\varepsilon}(x))-(e(x)+f(x)w_2^{\varepsilon}(x))w_2^{\varepsilon}(x),
		\end{align*}
		from which \eqref{l2-eq1} follows.
		
		\medskip
		
		{\rm (iii)} Suppose that $e_{\min}>0$. Thanks to the fact that $w_2^{\varepsilon}F^{\varepsilon}(\cdot,w_2^{\varepsilon})=0$, we have that 
		\begin{align*}
			0= \Big((h_+^{\varepsilon}+\varepsilon)-\frac{(\varepsilon+r(x))s(x)}{(a(x)+s(x))}-f(x)w_2^{\varepsilon}(x)\Big)w_2^{\varepsilon}(x)+(s(x)-g(x)w_2^{\varepsilon}(x))w_1^{\varepsilon}(x;w_2^{\varepsilon}).
		\end{align*}
		Hence
		\begin{align}
			s(x)-g(x)w_2^{\varepsilon}=&\Big(\frac{(\varepsilon+r(x))s(x)}{(a(x)+s(x))}+f(x)w_2^{\varepsilon}(x)-(h_+^{\varepsilon}+\varepsilon)\Big)\frac{w_2^{\varepsilon}(x)}{w_1^{\varepsilon}(x;w_2^{\varepsilon}(x))}\cr
			=&2\Big(\frac{(\varepsilon+r(x))s(x)}{(a(x)+s(x))}+f(x)w_2^{\varepsilon}(x)-(h_+^{\varepsilon}+\varepsilon)\Big)\frac{(\varepsilon+r(x))}{G^{\varepsilon}(x,w_2^{\varepsilon}(x))}.
		\end{align}
		If $h^{\varepsilon}(x)\le 0$, then
		\begin{align*}
			s(x)-g(x)w_2^{\varepsilon}(x)=&2\Big(\frac{(\varepsilon+r(x))s(x)}{(a(x)+s(x))}+f(x)w_2^{\varepsilon}(x)-\varepsilon\Big)\frac{(\varepsilon+r(x))}{G^{\varepsilon}(x,w_2^{\varepsilon}(x))}\cr
			\ge &2\Big(\frac{r_{\min}s_{\min}}{\|a+s\|_{\infty}}-\varepsilon\Big)\frac{(\varepsilon+r(x))}{G^{\varepsilon}(x,w_2^{\varepsilon}(x))}
			\ge \frac{r^2_{\min}s_{\min}}{\|a+s\|_{\infty}\|G^{1}(\cdot,\tau_2)\|_{\infty}}, 
		\end{align*}
		whenever $0<\varepsilon<\min\Big\{1,\frac{r_{\min}s_{\min}}{2\|a+s\|_{\infty}}\Big\}$.\\
		If $h^{\varepsilon}(x)\ge 0$, then 
		\begin{align*}
			s(x)-g(x)w_2^{\varepsilon}(x)=&2\Big(e(x)+f(x)w_2^{\varepsilon}(x)-\varepsilon\Big)\frac{(\varepsilon+r(x))}{G^{\varepsilon}(x,w_2^{\varepsilon}(x))}\cr
			\ge& \frac{e_{\min}r_{\min}}{\|G^1(\cdot;\tau_2)\|_{\infty}}\quad \text{whenever}\ 0<\varepsilon\le \min\Big\{1,\frac{e_{\min}}{2}\Big\} \varepsilon.
		\end{align*}
		We can now take $\tilde{\varepsilon}_{\tau_2}:=\min\Big\{1,\frac{r_{\min}s_{\min}}{2\|a+s\|_{\infty}},\frac{e_{\min}}{2}\Big\}$ and $\tilde{m}_{\tau_2}=\min\Big\{\frac{r^2_{\min}s_{\min}}{\|a+s\|_{\infty}\|G^{1}(\cdot,\tau_2)\|_{\infty}},\frac{r^2_{\min}s_{\min}}{\|a+s\|_{\infty}\|G^{1}(\cdot,\tau_2)\|_{\infty}}\Big\}$.
	\end{proof}
	
	\medskip
	
	\begin{lem}\label{l3} Let $\tau_2>0$ and $\tilde{\varepsilon}_{\tau_2}$ be as in Lemma \ref{l2}. Let  $\varepsilon_{\tau_2}$ be as in Lemma \ref{l1}. Set ${\varepsilon}_*=\min\{\tilde{\varepsilon}_{\tau_2},\varepsilon_{\tau_2}\}$. For every $0<\varepsilon<{\varepsilon}_*$, let ${\bf w }^{\varepsilon}(x)=(w_1^{\varepsilon}(x;w_{2}^{\varepsilon}(x)),w_2^{\varepsilon}(x))$ for all $x\in\bar{\Omega}$, where $w_2^{\varepsilon}$ is given by Lemma \ref{l2} and $w_1^{\varepsilon}$ is defined by \eqref{w-1-def}. Then there exists   $\tilde{\bf w}^{\varepsilon}\in [C^2(\bar{\Omega})]_{\vec{n}}^2:=\{(w_1,w_2)\in [C^2(\bar{\Omega})]^2 : \partial_{\vec{n}}w_i=0 \ \text{on}\ \partial\Omega\}$ satisfying 
		\begin{equation}\label{l3-eq1}
			\min\{(r-c\tilde{w}_1^{\varepsilon})_{\min}, (s-g\tilde{w}_2^{\varepsilon})_{\min}\}>0, \quad \tilde{\bf w}^{\varepsilon}\in [C^{++}(\bar{\Omega})]^2, \quad \text{and}\quad   \lim_{\varepsilon\to 0}\|\tilde{\bf w}^{\varepsilon}-{\bf w}^{\varepsilon}\|=0.
		\end{equation}
		Furthermore, for each $0<\varepsilon<\varepsilon_*$, there is $d^{\varepsilon}>0$ such that for any choice of diffusion rates $0<d_1,d_2<d^{\varepsilon}$, $\tilde{\bf w}^{\varepsilon}$ satisfies
		\begin{equation}\label{l3-eq2}
			\begin{cases}
				0\ge d_1\Delta \tilde{w}_1^{\varepsilon} +(r(x)-c(x)\tilde{w}_1^{\varepsilon})\tilde{w}_2^{\varepsilon}-(a+s+b\tilde{w}_1^{\varepsilon}) \tilde{w}_1^{\varepsilon} & x\in\Omega,\cr 
				0\ge d_2\Delta \tilde{w}_2^{\varepsilon} +(s-g\tilde{w}_2^{\varepsilon})\tilde{w}_1^{\varepsilon}-(e+f\tilde{w}_2^{\varepsilon})\tilde{w}_2^{\varepsilon} & x\in\Omega,\cr   0=\partial_{\vec{n}}\tilde{w}_1^{\varepsilon}=\partial_{\vec{n}}\tilde{w}_2^{\varepsilon} & x\in\partial\Omega.
			\end{cases}
		\end{equation} 
	\end{lem}
	\begin{proof} Let $\{e^{t\Delta}\}_{t\ge 0}$ denote the analytic $c_0$-semigroup generated by the Laplace operator on $C(\bar{\Omega})$ subject to the homogeneous Neumann boundary conditions on $\partial\Omega$. Then,  for every $z\in C(\bar{\Omega})$, $e^{t\Delta}z\in {\rm Dom}_{\infty}(\Delta)$ for all $t>0$ and $\|e^{tz}-z\|_{\infty}\to 0$ as $t\to0^+$. Furthermore, by the maximum principle for parabolic equations, $e^{t\Delta}z\in C^{++}(\bar{\Omega})$ for all $t>0$ whenever $z\in C^{+}(\bar{\Omega})\setminus\{0\}$.
		
		\medskip
		Let ${m}_*:=\min\{m_{\tau_2},\tilde{m}_{\tau_2}\}$, where $m_{\tau_2}$ and $\tilde{m}_{\tau_2}$ are given by Lemma \ref{l1}-{\rm (iii)} and Lemma \ref{l2}-{\rm (iii)}, respectively. Since for every $0<\varepsilon<\varepsilon_*$, $0<\tilde{w}_2^{\varepsilon}$, it follows from Lemma \ref{l1}-{\rm (iii)}  and Lemma \ref{l2}-{\rm (iii)} that 
		\begin{equation*}
			\min_{x\in\bar{\Omega}}(r-cw_1^{\varepsilon}(x;w_2^{\varepsilon}(x)))\ge m_{*} \quad \text{and}
			\quad 
			(s(x)-g(x)w_2^{\varepsilon}(x))\ge m_*.
		\end{equation*}
		Hence, for every $0<\varepsilon<\varepsilon_*$, we can choose $0<t_{\varepsilon}\ll 1$ such that $\tilde{w}^{\varepsilon}_1:=e^{t_{\varepsilon}\Delta}w_1^{\varepsilon}(\cdot;w_2^{\varepsilon}(\cdot))\in C^{++}(\bar{\Omega})$ and $\tilde{w}_2^{\varepsilon}:=e^{t_{\varepsilon}\Delta}w_2^{\varepsilon}\in C^{++}(\bar{\Omega})$ satisfy 
		\begin{equation*}
			\min_{x\in\bar{\Omega}}(r-c\tilde{w}_1^{\varepsilon}(x))\ge \frac{m_{*}}{2},
			\quad 
			\min_{x\in\bar{\Omega}}(s(x)-g(x)\tilde{w}_2^{\varepsilon}(x))\ge \frac{m_*}{2},\quad \text{and}\quad \|\tilde{\bf w}^{\varepsilon}-{\bf w}^{\varepsilon}\|<\varepsilon.
		\end{equation*}
		Furthermore, by \eqref{l1-eq1} and \eqref{l2-eq1}, possible after decreasing $0<t_{\varepsilon}$ we may suppose that 
		\begin{equation}\label{A3}
			\begin{cases}
				-\frac{\varepsilon}{3} \tilde{w}_2^{\varepsilon}(x)\ge (r(x)-c(x)\tilde{w}_1^{\varepsilon}(x))\tilde{w}_2^{\varepsilon}-(a(x)+s(x)+b(x)\tilde{w}_1^{\varepsilon}(x))\tilde{w}_1^{\varepsilon}(x) & x\in\bar{\Omega},
				\cr
				-\frac{\varepsilon}{3} \tilde{w}_2^{\varepsilon}(x)\ge (s(x)-g(x)\tilde{w}_2^{\varepsilon}(x))\tilde{w}_1^{\varepsilon}(x))-(e(x)+f(x)\tilde{w}_2^{\varepsilon}(x))\tilde{w}_2^{\varepsilon}(x) & x\in\bar{\Omega}.
			\end{cases}
		\end{equation}
		Finally, for every $0<\varepsilon<\varepsilon_*$, take
		$$
		d^{\varepsilon}:=
		\frac{\varepsilon\min_{x\in\bar{\Omega}}\tilde{w}_2^{\varepsilon}(x)}{3(1+\|\Delta\tilde{w}_1^{\varepsilon}\|_{\infty}+\|\Delta\tilde{w}_2^{\varepsilon}\|_{\infty})}.
		$$
		Then by \eqref{A3}, $\tilde{\bf w}^{\varepsilon}$ satisfies \eqref{l3-eq2} for every $0<d_1,d_2<d^{\varepsilon}$, which completes the proof of the lemma.  
	\end{proof}

	\begin{lem}\label{l4} Suppose that $(\frac{rs}{a+s}-e)_{\max}>0$. 
		Fix $0<\varepsilon<{\varepsilon}_*$, where ${\varepsilon}_*$ is  as in Lemma \ref{l3}. Let $\tilde{\bf w}^{\varepsilon}$ and $d^{\varepsilon}$ be as in Lemma \ref{l3}. Then there is $0<\tilde{d}^{\varepsilon}<d^{\varepsilon}$ such that for every $0<d_1,d_2<\tilde{d}^{\varepsilon}$, $\lambda_{d_1,d_2}>0$ and   system \eqref{Eq1} has a positive steady state ${\bf u}$ satisfying ${\bf 0}<< {\bf u}\le \tilde{\bf w}^{\varepsilon}$.

	\end{lem}
	\begin{proof} We proceed in three steps. Set $[{\bf 0},\tilde{\bf w}^{\varepsilon}]:=\{{\bf u}\in[C(\bar{\Omega})]^2 : {\bf 0}\le {\bf u}\le \tilde{\bf w}^{\varepsilon}\}$.
		\medskip
		
		{\bf Step 1.} Fix $0<d_1,d_2<d^{\varepsilon}$. We show that 
		\begin{equation}\label{A5}
			{\bf u}(t,\cdot;{\bf u}_0)\in[{\bf 0}, \tilde{\bf w}^{\varepsilon}]\quad \text{whenever} \quad {\bf u}_0\in[{\bf 0},\tilde{\bf w}^{\varepsilon}].
		\end{equation}
		To this end, fix ${\bf u}_0\in[{\bf 0},\tilde{\bf w}^{\varepsilon}]$ and set ${\bf v}(t,\cdot):=\tilde{\bf w}^{\varepsilon}-{\bf u}(t,\cdot;{\bf u}_0)$. Then, thanks to \eqref{l3-eq2}, ${\bf v}(t,\cdot)$ satisfies 
		\begin{align*}
			\partial_tv_1=& -d_1\Delta u_1-(r-cu_1)u_2+(a+s+bu_1)u_1\cr 
			\ge& d_1\Delta v_1+(r-cw_1^{\varepsilon})w_2^{\varepsilon}-(r-cu_1)u_2-((a+s+bw_1^{\varepsilon})w_1^{\varepsilon}-(a+s+bu_1)u_1)\cr 
			=&d_1\Delta v_1+(r-cw_1^{\varepsilon})v_2+((r-cw_1^{\varepsilon})u_2-(r-cu_1)u_2)-(a+s+b(u_1+w_1^{\varepsilon}))v_1\cr  
			=&d_1\Delta v_1+(r-cw_1^{\varepsilon})v_2-(a+s+b(u_1+w_1^{\varepsilon})+cu_2)v_1\quad \qquad\qquad\qquad x\in\Omega,\ t>0.
		\end{align*}
		Similarly
		\begin{equation*}
			\partial_tv_2\ge d_2\Delta v_2+(s-gw_1^{\varepsilon})v_1-(e+f(u_2+w_2^{\varepsilon})+gu_1)v_2 \quad x\in\Omega,\ t>0.
		\end{equation*}
		Note also that $\partial_{\vec{n}}v_1=\partial_{\vec{n}}v_2=0$ of $(0,\infty)\times\partial\Omega$. Therefore, since $(r-cw_{1}^{\varepsilon})>0$ and $(s-gw_2^{\varepsilon})>0$ on $\bar{\Omega}$ by \eqref{l3-eq1}, and ${\bf 0}\le {\bf v}(0,\cdot)$, we can employ the comparison principle for cooperative systems to conclude that ${\bf v}(t,\cdot)\ge {\bf 0}$ for all $t>0$, which completes the proof of \eqref{A5}.

		\medskip
		
		{\bf Step 2.} Fix $0<d_1,d_2<\tilde{d}^{\varepsilon}$. There is $0<\tilde{d}^{\varepsilon}\le d^{\varepsilon}$ such that 
		\begin{equation}\label{A6}
			\lambda_*>0\quad \forall\ 0<d_1,d_2<\tilde{d}^{\varepsilon}
		\end{equation}
		Indeed, since $(\frac{rs}{a+s}-e)_{\max}>0$, then $(rs-e(a+s))_{\max}>0$. Then by \cite[Proposition 1 ]{CCM2020}, there is $\tilde{d}>0$ such that $\lambda_*>0$ for every $0<d_1,d_2<\tilde{d}$. So, we can take $\tilde{d}^{\varepsilon}:=\min\{\tilde{d},d^{\varepsilon}\}$, so that \eqref{A6} holds whenever $0<d_1,d_2<\tilde{d}^{\varepsilon}$.
		
		\medskip
		
		{\bf Step 3.} Note that by the first inequality of \eqref{l3-eq1}, and Step 1, solution operator of system \eqref{Eq1} generates a strongly cooperative semi-flow  $\{\Phi(t)\}_{t\ge 0}$ on $[{\bf 0},\tilde{\bf w}^{\varepsilon}]$. Moreover, by the regularity theory for parabolic equations, $\Phi(t)$ is compact for every $t>0$. Moreover, by Step 2, the trivial steady state ${\bf u}:={\bf 0}$ is linearly  unstable, and by step 1, $\tilde{\bf w}^{\varepsilon}$ is super-solution of \eqref{Eq1}. Therefore, by the theory of monotone dynamical systems \cite{Hess}, there is ${\bf 0}<{\bf u}(\cdot;{\bf d})\le \tilde{\bf w}^{\varepsilon}$ such that $\Phi(t){\bf u}(\cdot;{\bf d})={\bf u}(\cdot;{\bf d})$ for all $t\ge 0$. Hence, ${\bf u}(\cdot;{\bf d})$ is a positive steady state solution of \eqref{Eq1}.
	\end{proof}
	
	Thanks to the above lemmas, we can now give a proof of Theorem \ref{TH5}.
	
	\begin{proof}[Proof of Theorem \ref{TH5}] Take $\varepsilon_0=\frac{\varepsilon_*}{2}$ and $d_0=\tilde{d}^{\varepsilon_0}$ where $\varepsilon_*$ and $\tilde{d}^{\varepsilon_0}$ is as in Lemma \ref{l4}.  Fix $0<d_1,d_2<d_0$. By Lemma \ref{l4}, system \eqref{Eq1} has positive steady state ${\bf u}(\cdot;{\bf d})$ satisfying ${\bf u}(\cdot;{\bf d})\le \tilde{\bf w}^{\varepsilon_0}$, where $\tilde{\bf w}^{\varepsilon_0}$ is as in Lemma \ref{l3}.  Therefore, by the first inequality in \eqref{l3-eq1}, we have that 
		$$
		(r-cu_1(\cdot;{\bf d}))_{\min}\ge (r-cw_1^{\varepsilon_0})_{\min}>0\quad \text{and}\quad (s-gu_2(\cdot;{\bf d}))_{\min}\ge (s-gw_2^{\varepsilon_0})_{\min}>0.
		$$
		Therefore, ${\bf u}(\cdot;{\bf d})$ satisfies hypothesis  {\bf (H2)}. Therefore, by Theorem \ref{TH4}, ${\bf u}(\cdot;{\bf d})$ is globally stable with respect to positive perturbations.
	\end{proof}

	\section{Proof of Theorem \ref{TH6}}
	
	As in the previous section, we shall always suppose that {\bf (H1)} and {\bf (H3)} hold. We prove a few lemmas. Throughout the rest of this section, for every $0<\varepsilon<\varepsilon_*$, where $\varepsilon_*$ is given by Lemma \ref{l3}, ${\bf w}^{\varepsilon}$ is as in Lemma \ref{l2} and $\tilde{\bf w}^{\varepsilon}$ is as in Lemma \ref{l3}.
	
	\begin{lem}\label{l5} 
		For every $x\in\bar{\Omega}$, let $w_2^0(x)$ denote the unique nonnegative solution of the algebraic equation in $\tau\ge 0$ of
		\begin{equation}
			F^{0}(x,\tau)=0,
		\end{equation}
		where $F^{\varepsilon}$, $\varepsilon\ge 0$ is defined by \eqref{tilde-F-funct-def}. Define also $w_1^0(x)=w_1^0(x;w_2^{0}(x))$ for every $x\in\bar{\Omega}$, where $w_1$ is defined by \eqref{w-1-def}. Then ${\bf w}^{0}:=(w_1^0,w_2^0)\in [C^{+}(\bar{\Omega})]^2$, $\|{\bf w}^{\varepsilon}-{\bf w}^{0}\|\to 0$ as $\varepsilon\to 0^+$ , where ${\bf w}^{\varepsilon}$ is as in Lemma \ref{l2} for every $\varepsilon>0$. In particular $\|\tilde{\bf w}^{\varepsilon}-{\bf w}^0\|\to 0$ as $\varepsilon\to 0^+$.  Furthermore, when $x\in\bar{\Omega}$ such that $h^{0}(x)>0$, where $h^{\varepsilon}$ is defined in \eqref{h-def-eq}, we have that $\min\{w_1^0(x),w_2^0(x)\}>0$ and 
		\begin{equation}\label{l5-eq1}
			\begin{cases}
				0=(r(x)-c(x)w_1^0(x))w_2^0(x)-(a(x)+s(x)+b(x)w_1^0(x))w_1^0(x),\cr
				0=(s(x)-g(x)w_2^0(x))w_1^0(x)-(e(x)+f(x)w_2^0(x))w_2^0(x),
			\end{cases}
		\end{equation}
		whereas $w_1^0(x)=w_2^0(x)=0$ when $h^0(x)\le 0$.
		
	\end{lem}
	\begin{proof} The continuity of $w_2^0$  and the fact that $\|w_2^{\varepsilon}-w_2^0\|_{\infty}\to0$ as $\varepsilon\to 0^+$ follows from the implicit function theorem. This in turn implies that $w_1^0\in C(\bar{\Omega})$ and a composition of two such functions, and $\|w_1^{\varepsilon}-w_1^0\|_{\infty}\to 0$ as $\varepsilon\to 0^+$. Clearly, we have that $w_2^0\ge 0$, which thanks to \eqref{w-1-def} implies that $w_1^0\ge 0$. Thus $\|{\bf w}^{\varepsilon}-{\bf w}^0\|\to 0$ as $\varepsilon\to 0^+$. This in turn with the fact that $\|\tilde{\bf w}^{\varepsilon}-{\bf w}^{\varepsilon}\|\to 0$ as $\varepsilon\to 0^+$ (see \eqref{l3-eq1}) implies that $\|\tilde{\bf w}^{\varepsilon}-{\bf w}^0\|\to 0$ as $\varepsilon\to 0^+$. Clearly, $w_2^0(x)=w_1^0(x)=0$ whenever $h^0(x)=0$.
		
		\medskip
		
		Next, fix $x\in\bar{\Omega}$ such that $h^0(x)>0$.  Then,  $ 
		F^{0}(x,0)=h^0(x)>0$. Hence, since $F^0(x,\tau_2)<0$ (where $\tau_2$ is as in Lemma \ref{l2}), it follows from the intermediate value theorem, the fact that $F^0(x,\tau)$ is strictly decreasing in $\tau\ge 0$, and the fact that $F^{0}(x,w_2^0(x))=0$ that $w_2^0(x)>0$. This along with \eqref{w-1-def} implies that $w_1^0(x)>0$. The first equation of \eqref{l5-eq1} follows from \eqref{l1-eq1} with $\varepsilon=0$ and $\tau=w_2^0(x)$. Furthermore, since $h^{0}(x)>0$, then it follows from the proof of \eqref{l2-eq2} that equality holds in this case with $\varepsilon=0$, which yields the validity of the second equation of \eqref{l5-eq1}.
	\end{proof}

	Next, for each $0<\varepsilon<\varepsilon_*$, consider the linear eigenvalue problem of the cooperative system
	\begin{equation}\label{A8}
		\begin{cases}
			\lambda^{\varepsilon}v_1=d_1\Delta v_1+(r-c\tilde{w}_1^{\varepsilon})v_2 -(a+s)v_1 & x\in\Omega,\cr 
			\lambda^{\varepsilon}v_2=d_2\Delta v_2 +(s-g\tilde{w}_2^{\varepsilon})v_1 -ev_2 & x\in\Omega,\cr 
			0=\partial_{\vec{n}}v_1=\partial_{\vec{n}}v_2 & x\in\partial\Omega.
		\end{cases}
	\end{equation}
	Denote by $\lambda^{\varepsilon}_*$ the principal eigenvalue of \eqref{A8}.  Note from \eqref{l5-eq1} that for every $x\in\bar{\Omega}$ satisfying $h^0(x)>0$, that is $r(x)s(x)>e(x)(a(x)+s(x))$, we have that 
	\begin{align*}
		&(r(x)-c(x)w_1^0(x))(s(x)-g(x)w_2^0(x))-(a(x)+s(x))e(x)\cr
		=&(a(x)+s(x))w_2^0(x)+e(x)b(x)w_1^0(x)+f(x)b(x)w_1^0(x)w_2^{0}(x)>0.
	\end{align*}
	Hence $((r-cw_1^0)(s-gw_2^0)-(a+s)e)_{\max}>0$. Therefore, since $\|\tilde{\bf w}^{\varepsilon}-{\bf w}^0\|\to 0$ as $\varepsilon\to 0^+$ (see Lemma \ref{l5}), then there is $0<\tilde{\varepsilon}_*<\varepsilon_*$ such that $((r-c\tilde{w}_1^{\varepsilon})(s-g\tilde{w}_2^{\varepsilon})-(a+s)e)_{\max}>0$ for $0\le \varepsilon<\tilde{\varepsilon}_*$. Therefore, by \cite[Proposition 1]{CCM2020}, for every $0<\varepsilon<\tilde{\varepsilon}_*$, there is $\tilde{d}^{\varepsilon}_{*}>0$ such that $\tilde{\lambda}^{\varepsilon}_*>0$ for every $0<d_1,d_2<\tilde{d}^{\varepsilon}_*$.
	
	\medskip
	
	The following lemma is needed.
	
	\begin{lem} Fix $0<\varepsilon<\tilde{\varepsilon}_*$ and $0<d_1,d_2<\tilde{\varepsilon}_*$. Then there is a unique positive steady state ${\bf v}(\cdot;{\bf d})$ to the cooperative system 
		\begin{equation}\label{A10}
			\begin{cases}
				0=d_1\Delta v_1+(r-c\tilde{w}_1^{\varepsilon})v_2-(a+s+bv_1)v_1 & x\in\Omega,\cr
				0=d_2\Delta v_2+(s-g\tilde{w}_2^{\varepsilon})v_1 - (e+fv_2)v_2 &  x\in\Omega,\cr 
				0=\partial_{\vec{n}}v_1=\partial_{\vec{n}}v_2 & x\in\partial\Omega.
			\end{cases}
		\end{equation}
		Furthermore, ${\bf v}(\cdot;{\bf d})\to {\bf v}^{\varepsilon}$ as $\max\{d_1,d_2\}\to 0^+$ locally uniformly in $\Omega$, where for each $x\in\bar{\Omega}$, ${\bf v}^{\varepsilon}(x)$ is the unique nonnegative stable solutions of the system algebraic equations 
		\begin{equation}\label{A9}
			\begin{cases}
				0=(r-c\tilde{w}_1^{\varepsilon})v_2^{\varepsilon}-(a+s+bv_1^{\varepsilon})v_1^{\varepsilon} & x\in\Omega,\cr
				0=(s-g\tilde{w}_2^{\varepsilon})v_1^{\varepsilon} - (e+fv_2^{\varepsilon})v_2^{\varepsilon} &  x\in\Omega. 
			\end{cases}
		\end{equation}
		
	\end{lem}
	\begin{proof} Since $\lambda^{\varepsilon}_*>0$ for every $0<d_1,d_2<\tilde{\varepsilon}_*$ and system \eqref{A9} is cooperative, strictly subhomogenous, and nonnegative classical solutions of the corresponding parabolic system are eventually bounded, then it has a unique positive steady state solution ${\bf v}(\cdot;{\bf d})$.   The asymptotic profiles of ${\bf v}(\cdot;{\bf d})$ as $\max\{d_1,d_2\}\to 0$ follows from \cite[Theorem 1]{CCM2020}.   
	\end{proof}
	
	We complete this section with a proof of Theorem \ref{TH6}.
	
	\begin{proof}[Proof of Theorem \ref{TH6}] By Lemma \ref{l4}, we have that ${\bf 0}\ll {\bf u}(\cdot;{\bf d})\le \tilde{\bf w}^{\varepsilon}$ for every $0<d_1,d_2<\min\{\tilde{d}^{\varepsilon},\tilde{d}^{\varepsilon}_*\}$, $0<\varepsilon<\tilde{\varepsilon}_*$. Hence, since by Lemma \ref{l5}, $\|\tilde{\bf w}^{\varepsilon}-{\bf w}^0\|\to 0$ as $\varepsilon\to 0^+$, then 
		\begin{equation}\label{A12}
			\limsup_{\max\{d_1,d_2\}}{\bf u}(x;{\bf d})\le {\bf w}^0(x) \quad \text{for }\ x\ \  \text{uniformly in}\ \bar{\Omega}.
		\end{equation}
		Next, for every $0<\varepsilon<\tilde{\varepsilon}_*$ and $0<d_1,d_2<\min\{\tilde{d}^{\varepsilon},\tilde{d}^{\varepsilon}\}$, since ${\bf u}(\cdot;{\bf d})\le \tilde{w}^{\varepsilon}$ and ${\bf u}(\cdot;{\bf d})$ is a positive steady state solution of \eqref{Eq1}, then 
		$$
		\begin{cases}
			0\ge d_1\Delta u_1(\cdot;{\bf d})+(r-c\tilde{w}_1^{\varepsilon})u_2(\cdot;{\bf d})-(a+s+bu_1(\cdot;{\bf d})){u}_1(\cdot,{\bf d}) & x\in\Omega,\cr
			0\ge d_2\Delta u_2(\cdot;{\bf d})+(s-g\tilde{w}_2^{\varepsilon})u_1(\cdot;{\bf d})-(e+fu_2(\cdot;{\bf d})){u}_2(\cdot;{\bf d}) & x\in\Omega,\cr
			0=\partial_{\vec{n}}u_1(\cdot;{\bf d})=\partial_{\vec{n}}u_2(\cdot;{\bf d}) & x\in\partial\Omega,
		\end{cases}
		$$
		and hence ${\bf u}(\cdot;{\bf d})$ is super-solution of \eqref{A10}. Therefore, ${\bf v}(\cdot;{\bf d})\le {\bf u}(\cdot;{\bf d})$ for every $0<d_1,d_2<\min\{\tilde{d}^{\varepsilon}_*,\tilde{d}^{\varepsilon}\}$ and $0<\varepsilon<\tilde{\varepsilon}_*$.  Hence, letting $\max\{d_1,d_2\}\to0^+$, we deduce from Lemma \ref{l5} that 
		\begin{equation}\label{A13}
			{\bf v}^{\varepsilon}(x)\le \liminf_{\max\{d_1,d_2\}\to 0^+}{\bf u}(x;{\bf d})\quad \text{for }\ x\ \text{locally uniformly in}\ \Omega.
		\end{equation}
		where ${\bf v}^{\varepsilon}$, $0<\varepsilon<\min\{\tilde{d}^{\varepsilon}_*,\tilde{d}^{\varepsilon}\}$, is the unique nonnnegative stable solution of \eqref{A9}. Since $\|\tilde{\bf w}^{\varepsilon}-{\bf w}^0\|\to 0$ as $\varepsilon\to 0$ and ${\bf w}^0$ is the unique nonnegative stable  solution of the system of algebraic equations \eqref{l5-eq1}, then letting $\varepsilon\to 0$ in \eqref{A9}, we have that $\|{\bf v}^{\varepsilon}-{\bf w}^0\|\to 0$ as $\varepsilon\to 0^+$.  Thus, sending $\varepsilon\to0^+$ in \eqref{A13}, we have that
		\begin{equation}\label{A14}
			{\bf w}^{0}(x)\le \liminf_{\max\{d_1,d_2\}\to 0^+}{\bf u}(x;{\bf d})\quad \text{for }\ x\ \text{locally uniformly in}\ \Omega.
		\end{equation}
		Observing that ${\bf w}^0(x)$ is the unique nonnegative stable solution of \eqref{TH6-eq2} for every $x\in\bar{\Omega}$, then ${\bf w}^0={\bf u}^*$. Therefore \eqref{TH2-eq1} follows from \eqref{A12} and \eqref{A14}.
	\end{proof}

	{\bf Acknowledgment.} The authors are grateful to the anonymous reviewers for their valuable comments and suggestions which help improved the presentation of the paper. The authors acknowledge support from IMU and GRAID program.

	\subsection*{Declarations}
	{\bf Ethical Approval:} Not applicable for this study.
	
	\vspace{0.05in}
	
	\noindent{\bf Competing interests:} The authors declare that there is no competing interest.
	
	\vspace{0.05in}
	
	\noindent{\bf Authors' contributions:} We acknowledge that all the three authors contributed equally in designing and conducting the study and in the mathematical analysis of the results.
	
	\vspace{0.05in}

	\vspace{0.05in}
	
	\noindent{\bf Availability of data and materials:} Not applicable.

\end{document}